\documentclass[11pt,reqno,a4paper]{amsart}

\usepackage{graphicx}
\usepackage{color}
\usepackage{transparent}

\usepackage[latin1]{inputenc}  %FANTASTICO PER ACCENTI IN ITALIANO

\usepackage{mathrsfs}
\usepackage{amssymb,amsmath, amsfonts, amsthm}
\usepackage{latexsym}
\usepackage[dvips]{epsfig}
\newtheorem{theorem}{Theorem}[section]
\newtheorem{proposition}[theorem]{Proposition}
\newtheorem{lemma}[theorem]{Lemma}
\newtheorem{corollary}[theorem]{Corollary}
\newtheorem{remark}[theorem]{Remark}



\newcommand{\ble}{\begin{lemma}}
\newcommand{\ele}{\end{lemma}}
\newcommand{\be}{\begin{equation*}}
\newcommand{\ee}{\end{equation*}}

\newcommand{\bel}{\begin{equation}}
\newcommand{\eel}{\end{equation}}

\newcommand{\ep}{\varepsilon}
\newcommand{\fr}{\frac }

\newcommand{\lap}{\Delta}
\newcommand{\N}{\mathbb{N}}

\newcommand{\R}{\mathbb{R}}

\renewcommand{\to}{\rightarrow}

\newcommand{\upp}{u_p}

\def\sideremark#1{\ifvmode\leavevmode\fi\vadjust{\vbox to0pt{\vss% the remark
 \hbox to 0pt{\hskip\hsize\hskip1em%                          will appear only
 \vbox{\hsize2.1cm\tiny\raggedright\pretolerance10000%          on the side
  \noindent #1\hfill}\hss}\vbox to15pt{\vfil}\vss}}}%

\begin{document}

\numberwithin{equation}{section}
\parindent=0pt
\hfuzz=2pt
\frenchspacing

\title[]{Exact Morse index computation for nodal radial solutions of Lane-Emden problems}

\author[]{Francesca De Marchis, Isabella Ianni, Filomena Pacella}

\address{Francesca De Marchis, University of Roma {\em Sapienza}, P.le Aldo Moro 8, 00185 Roma, Italy}
\address{Isabella Ianni, Seconda Universit\`a degli Studi di Napoli, V.le Lincoln 5, 81100 Caserta, Italy}
\address{Filomena Pacella, University of Roma {\em Sapienza}, P.le Aldo Moro 8, 00185 Roma, Italy}



\thanks{Research partially supported by: FIRB project \textquotedblleft{Analysis and Beyond}\textquotedblright, PRIN $201274$FYK7$\_005$ grant and INDAM - GNAMPA}

\begin{abstract} We consider the semilinear Lane-Emden problem
\begin{equation}\label{problemAbstract}\left\{\begin{array}{lr}-\Delta u= |u|^{p-1}u\qquad  \mbox{ in }B\\
u=0\qquad\qquad\qquad\mbox{ on }\partial B
\end{array}\right.\tag{$\mathcal E_p$}
\end{equation}
where $B$ is the unit ball of $\R^N$, $N\geq2$, centered at the origin and $1<p<p_S$, with $p_S=+\infty$ if $N=2$ and $p_S=\frac{N+2}{N-2}$ if $N\geq3$.
Our main result is to prove that in dimension $N=2$ the Morse index of the least energy sign-changing radial solution $u_p$ of \eqref{problemAbstract} is exactly $12$ if $p$ is sufficiently large. As an intermediate step we compute explicitly  the first eigenvalue of a limit weighted problem in $\R^N$ in any dimension $N\geq2$.
\end{abstract}

\maketitle

\section{Introduction}\label{Introduction}

\

We consider the classical Lane-Emden problem
\begin{equation}\label{problem}\left\{\begin{array}{lr}-\Delta u= |u|^{p-1}u\qquad  \mbox{ in }B\subset\R^N\\
u=0\qquad\qquad\qquad\mbox{ on }\partial B
\end{array}\right.
\end{equation}
where $B$ is the unit ball of $\R^N$, $N\geq 2$, centered at the origin and $1<p<p_S$, with $p_S=+\infty$ if $N=2$ and $p_S=2^*-1=\frac{N+2}{N-2}$ if $N\geq3$.\\
It is well know that, due to the oddness of the nonlinearity, \eqref{problem} admits infinitely many solutions. In particular exactly two of them have constant sign and are radial, while all the others change sign. Among these ones, one can select the least energy sign changing solution whose existence can be proved by minimizing the associated energy functional on the \emph{nodal Nehari set} in the space $H^1_0(B)$, exploiting the subcriticality of the exponent $p$ (see \cite{CCN} and \cite{BartschWeth} for details). Several properties of these \emph{minimal} solutions can be proved, in particular they have only two nodal regions and their Morse index is precisely two. We recall that the \emph{Morse index} $m(u)$ of a solution $u$ of \eqref{problem} is the maximal dimension of a subspace $X\subset H^1_0(B)$ where the quadratic form associated to the linearized operator at $u$
\[
L_u=(-\lap-p|u|^{p-1})
\]
is negative definite. Equivalently, since $B$ is a bounded domain, $m(u)$ can be defined as the number of the negative eigenvalues of $L_u$ counted with their multiplicity.

By doing the same minimizing procedure on the nodal Nehari set in the Sobolev space of radial functions $H^1_{0,rad} (B)$ one ends up with a least energy radial sign changing solution $u_p$ of \eqref{problem} whose \emph{radial} Morse index, i.e. in the space $H^1_{0,rad} (B)$, is precisely $2$.

For some time it was an open question to establish whether the least energy radial nodal solution $u_p$ was the least energy nodal solution in the whole space $H^1_0(B)$ or not. This question was answered in \cite{AftalionPacella} by showing, for general semilinear elliptic problems with autonomous nonlinearities, that radial nodal solutions, in balls or annuli, have Morse index greater than or equal to $N+2$ (see Lemma 3.1), so they cannot be the least energy nodal solutions.

As a consequence the question of estimating or computing the Morse index $m(u_p)$ of the least energy nodal radial solution $u_p$ in the whole space $H^1_0(B)$ raised.\\
In this paper we analyze this problem and our main result is the computation of $m(u_p)$, in dimension $N=2$, for large exponents. More precisely we have:

\begin{theorem}\label{teoPrincipale}
Let $N=2$ and $u_p$ be the least energy sign-changing radial solution to \eqref{problem}. Then
\[
m(u_p)=12 \qquad\mbox{ for $p$ sufficiently large }
\]
where $m(u_p)$ is the Morse index of $u_p$ in $H^1_0(B)$.
\end{theorem}

Let us explain how we achieve the result and why we get it in the two dimensional case and for large exponents $p$.

Since our solution $u_p$ is radial, to study the spectrum of the linearized operator $L_p:=L_{u_p}$ a suitable procedure could be to decompose it as a sum of the spectrum of a radial weighted operator and the spectrum of the Laplace-Beltrami operator on the unit sphere. This works well when the domain is an annulus (see for example \cite{BartschClappGrossiPacella} and \cite{GladialiGrossiPacellaSrikanth}) but leads to a weighted eigenvalue problem with a singularity at the origin if the domain is a ball. To bypass this difficulty we first approximate the ball $B$ by annuli $A_n$ with a small hole, showing that the number of negative eigenvalues of the linearized operator $L_p$ is preserved (see Section \ref{section:linearizedOperator}).\\
Then the computation of the Morse index of $L_p$ in $B$ corresponds to estimate the eigenvalues of the operator
\[
\widetilde{L_{p}^n} = |x|^2 \left(  -\Delta -V_p(x) \right)
\]
in $H^1_0(A_n)$, where the potential $V_p(x)$ is $p|u_p(x)|^{p-1}$ (see Section \ref{section:auxiliaryweightedoperator}). In particular it turns out that the Morse index of $u_p$ is determined mainly by the \emph{size} of the first (radial) eigenvalue $\widetilde{\beta_1}(p)$ of this operator, with $n=n_p$ fixed properly.

In order to study this eigenvalue, a \emph{good knowledge} of the potential $V_p(x)$ is needed, in other words this means to know qualitative properties of the solution $u_p$ of \eqref{problem}.
Here is where the hypotheses on the dimension and on the exponent $p$ enter.

Recently, in the paper \cite{GrossiGrumiauPacella2}, a very accurate analysis of the asymptotic behavior of the least energy radial nodal solution $u_p$ of \eqref{problem} in the ball in dimension $N=2$ has been done, as the exponent $p$ tends to infinity.\\
In particular it has been shown that a suitable rescaling of the positive part $u^+_p$ (assuming $u_p(0)>0$) converges to a regular solution of the Liouville problem in $\R^2$, while a suitable rescaling of the negative part $u^-_p$ converges to a solution of a singular Liouville problem in $\R^2$ (see also \cite{DeMachisIanniPacellaJEMS} for more general symmetric domains).

This allows to detect precisely the asymptotic behavior as $p\to+\infty$ of the \emph{crucial} eigenvalue $\widetilde{\beta_1}(p)$ by several nontrivial estimates (see Section \ref{section:analysis}). Let us point out that the results in Section \ref{section:analysis}, in particular Lemma \ref{lemma:primastimabetapN=2}, show clearly that the contribution to the Morse index of $u_p$ comes mainly from the negative nodal region of $u_p$. It is interesting also to observe the relation between the value of $m(u_p)$ obtained in Theorem \ref{teoPrincipale} and the value of the Morse index of the radial solution of the singular Liouville problem in the whole plane which has been computed in \cite{ChenLin} (and also in \cite{GladialiGrossiNeves}), see  Remark \ref{remarkRelation} ahead.

\

The asymptotic analysis fulfilled  in \cite{GrossiGrumiauPacella2} and \cite{DeMachisIanniPacellaJEMS} allows also to prove a peculiar blow-up (in time) behavior of the solutions of the associated parabolic problem with initial data close to these nodal stationary solutions, for $p$ sufficiently large (\cite{DeMarchisIanni, DicksteinPacellaSciunzi}).

\

In the case of higher dimensions, $N\geq3$, such an accurate asymptotic analysis of $u_p$, as $p\to p_S$ is not yet available. Indeed the results of \cite{BAEMP}, where low energy nodal solutions of almost critical problems are studied, do not allow to carry on all the estimates needed to compute the limit of $\widetilde{\beta_1}(p)$, as $p\to p_S=\frac{N+2}{N-2}$. Therefore the study of the case $N\geq 3$ needs to be considered separately (see \cite{DeMarchisIanniPacellaNgeq3}).

\

Finally, let us point out that another important step for the proof of Theorem \ref{teoPrincipale} is to compute the first eigenvalue of the \emph{limit} weighted operator
\[
\widetilde L^* =|x|^{2}\left[-\Delta -V(x)\right],\qquad x\in\mathbb R^N
\]
with $V$ defined as in \eqref{limiteV}. This is done in Section \ref{section:limitWeighted} in every  dimension $N\geq2$ and we believe that the result could be useful also for other problems.
\tableofcontents

\

\section{Preliminary results in dimension $N=2$} \label{section:preliminaries}

\

In this section we state previous results about the asymptotic behavior of nodal solutions of \eqref{problem} in dimension $N=2$.
We start by recalling the following well known qualitative properties for radial least energy nodal solutions (which actually hold in any dimension $N\geq 2$):
\begin{proposition}\label{PropositionUnicoMaxeMin}  Let $(u_p)$ be a family of least energy radial nodal solutions to \eqref{problem} with $u_p(0)>0$, then:
\begin{itemize}
\item[$(i)$] $u_p$ has exactly $2$ nodal regions
\item[$(ii)$] $u_p(0)=\|u\|_{\infty}$
\item[$(iii)$] in each nodal region there is exactly one critical point (namely the maximum and the minimum points)
\end{itemize}
\end{proposition}
%\begin{proof} \edz{\textcolor{red}{decidere cosa fare della proof, toglierla proprio e mettere qualche referenza??}}
%\textcolor{blue}{$(i)$ derives from the minimality property of $u_p$. While $(i)$ and $(ii)$ follow by o.d.e. arguments. Indeed by  $(i)$ we know that there exist $r_p\in (0,1)$ such that, writing with abuse of notation $u_p(r)=u_p(|x|)$, then $\upp(r_p)=0$,  $u_p(r)>0$ for $r\in (0,r_p)$ and $u_p(r)<0$ for $r\in (r_p,1)$. Moreover, since $u_p$ is a positive solution of $-\lap \upp=|\upp|^{p-1}\upp$ in the ball $B_{r_p}(0)$ with $u_p=0$ on $\partial B_{r_p}(0)$, then it is known that
%$\upp(r)$ is strictly decreasing for $r\in (0,r_p)$ and also by the Hopf Lemma that $\frac{\partial u_p}{\partial r}(r_p)<0$, as a consequence $(ii)$ holds. To conclude the proof just observe that writing \eqref{problem} in polar coordinates we have that
%\[
%-\upp''(r)-\frac{N-1}{r}\upp'(r)=|\upp(r)|^{p-1}\upp(r)<0\qquad r\in(r_p,1),
%\]
%so it is easy to see that all the critical points of $\upp$ in $(r_p,1)$ should be minima and in turn this implies that $\upp$ admits a unique minimum point.}
%\end{proof}
From now on we will denote by $r_p$ the unique \emph{nodal radius} of $u_p$ and by $s_p$ the unique \emph{minimum radius} of $u_p$ i.e., writing with abuse of notation $u_p(r)=u_p(|x|)$, 
\begin{equation} \label{rp}
r_p\in (0,1)\ \mbox{ is such that }\ u_p(r_p)=0
\end{equation} 
and
\begin{equation}\label{sp}
s_p\in (r_p,1)\ \mbox{ is such that }\ \|u_p^-\|_{\infty}=u_p^-(s_p)=-u_p(s_p),
\end{equation}
where $u_p^-$ is the negative part of $u_p$.
\\
\\
Next we recall the results obtained in \cite{GrossiGrumiauPacella2} for least energy radial nodal solutions that we summarize in the following theorem.

\

\begin{theorem}\label{thm:GGP2} Let $N=2$ and let $(u_p)$ be a family of least energy radial nodal solutions to \eqref{problem} with $u_p(0)>0$. Let us define
\begin{eqnarray}\label{epsilon+-}
(\varepsilon_p^+)^{-2}&:=&
p u_p(0)^{p-1},\nonumber \\
(\varepsilon_p^-)^{-2}&:=&
p u_p(s_p)^{p-1},
\end{eqnarray}
and the rescaled functions
\begin{eqnarray}\label{zp}
&& z_p^+(x):=p\frac{u_p(\varepsilon_p^+ x)-u_p(0)}{u_p(0)}, \ \  x\in\frac{B}{\varepsilon_p^+}
\\\label{zm}
&& z_p^-(x):=p\frac{u_p(\varepsilon_p^- x)-u_p(s_p)}{u_p(s_p)} , \ \  x\in\frac{B}{\varepsilon_p^-}.
\end{eqnarray}
Then
\begin{eqnarray}
&&\varepsilon_p^{\pm}\underset{p\rightarrow +\infty}{\longrightarrow} 0\\
\label{epsilonpmN2}
&&z_p^+ \underset{p\rightarrow +\infty}{\longrightarrow} U\quad\mbox{in $C^1_{loc}(\R^2)$}
\label{zppiu2}
\\
&&z_p^-  \underset{p\rightarrow +\infty}\longrightarrow Z_{\ell}\quad\mbox{in $C^1_{loc}(\R^2\setminus\{0\})$}
\label{zpmeno2}
\end{eqnarray}
where
\begin{equation}\label{U}
    U(x):=\log\left(\fr1{1+\fr18 |x|^2}\right)^2
    \end{equation}
is the regular solution of
\begin{equation}
\label{LiouvilleEquation}
\left\{
\begin{array}{lr}
-\Delta U=e^U\quad\mbox{ in }\R^2\\
\int_{\R^2}e^Udx= 8\pi, U(0)=0
\end{array}
\right.
\end{equation} 
and
\begin{equation}\label{Zell}
Z_\ell(x):=\log\left(\frac{2(\gamma+2)^2\delta^{\gamma+2}|x|^{\gamma}}{(\delta^{\gamma+2}+|x|^{\gamma+2})^2}\right),
\end{equation}
with
\begin{equation}\label{varie}
\ell=\lim_{p\to+\infty}\frac{s_p}{\varepsilon^-_p}\approx 7.1979,\qquad \gamma=\sqrt{2\ell^2+4}-2,\qquad \delta=(\frac{\gamma+4}{\gamma})^{\frac{1}{\gamma+2}}\ell,
\end{equation}
is a singular radial solution of
\begin{equation}
\label{LiouvilleSingularEquation}
\left\{
\begin{array}{lr}
-\Delta Z=e^Z+ H\delta_{0}\quad\mbox{ in }\R^2\\
\int_{\R^2}e^Zdx<\infty
\end{array}
\right.
\end{equation}
where $H=-\int_0^\ell e^{Z_\ell(s)}s\,ds$ and $\delta_{0}$ is the Dirac measure centered at $0$.

Moreover if we denote by $r_p$ the nodal radius of $\upp$, then
\begin{eqnarray}\label{quindi:rapportoepsilon+-}
&&\frac{r_p}{\varepsilon^+_p}\underset{p\rightarrow +\infty} {\longrightarrow}+\infty,\qquad \frac{\varepsilon^-_p}{r_p}\underset{p\rightarrow +\infty} {\longrightarrow}+\infty.\end{eqnarray}
\end{theorem}

\begin{remark}\label{remarkRelation}
Note that it is the precise value of the constant $\ell$ (see \eqref{varie}) that allows in \cite{GrossiGrumiauPacella2} to determine the unique radial solution $Z_{\ell}$ of the singular Liouville problem to which $z_p^-$ converges.
As shown in \cite{ChenLin}, the Morse index of $Z_{\ell}$ is
\[m(Z_{\ell})=1+2\left[\frac{\sqrt{2\ell^2+4}}{2}\right]=11,\]
(where $\left[x\right]$ denotes the biggest integer which is less or equal than $x$), and the kernel of the linearized operator at $Z_{\ell}$ has dimension
\[k(Z_{\ell})=1.\]

Also in our proof (see Section \ref{section:analysis}) it is crucial to know the exact value of $\ell$ in order to prove that  $m(u_p)$ is precisely $12$.
The fact that
\[m(u_p)=m(Z_{\ell})+k(Z_{\ell}),\]
seems to indicate a connection between the spectrum of the linearized operator at $u_p$ and that of the linearized operator at $Z_{\ell}$. This stresses once again that the relevant contribution to the Morse index of $u_p$ is given by its negative nodal region.
\end{remark}

For more general symmetric domains, as a consequence of a general profile decomposition theorem, in the paper \cite{DeMachisIanniPacellaJEMS} further asymptotic results have been obtained. In particular we recall the following estimate that we will need later, which corresponds to property $(P_3^k)$ in  \cite[Proposition 2.2]{DeMachisIanniPacellaJEMS} (indeed in the radial case the origin is the only absolute maximum point of $|u_p|$ and $k=1$ by \cite[Proposition 3.6]{DeMachisIanniPacellaJEMS}):
\begin{equation}\label{Q3}
p|y|^2|u_p(y)|^{p-1}\leq C \quad \mbox{ for any $y\in B$.}%  \edz{\textcolor{red}{ho  tolto ''and $p$ sufficiently large'', perche vale per ogni p}}
\end{equation}

\

\

\

\section{Linearized operator and approximation of its eigenvalues}\label{section:linearizedOperator}

\

Let $u_p$ be a solution to \eqref{problem} and let $L_{p}: H^2(B)\cap H^1_0(B)\rightarrow L^2(B)$ be the linearized operator at $u_p$, namely
\begin{equation}\label{linearizedOperator} L_{p} v: =   -\Delta v-p|u_p(x)|^{p-1}v.
\end{equation}
It is well known that $L_p$ admits a sequence of eigenvalues which, counting them according to their multiplicity, we denote by
% \[\mu_i(p)\in\mathbb R, \quad i\in\mathbb N\]  namely
%\[L_pv=\mu_i(p)v\  \mbox{ for a certain }v\in H^1_0(B),\ v\neq 0\]
%and that
\[\mu_1(p)< \mu_2(p)\leq\ldots\leq\mu_i(p)\leq\ldots,\quad \mu_i(p)\rightarrow +\infty  \mbox{ as }i\rightarrow +\infty.\]
We also recall their min-max characterization
\begin{eqnarray}\label{CourantCharEigenv}
\mu_i(p) &=& \inf_{\substack{
W\subset H^1_{0}(B)\\ dim W=i}}   \max_{\substack{v\in W\\v\neq 0}}\ \ \
R_p[v],\qquad i\in\N^+
\end{eqnarray}
where $R_p[v]$ is the Rayleigh quotient
\begin{equation}\label{Rayleigh}
R_p[v]:=\frac{Q_p(v)}{\int_B v(x)^2 dx}
\end{equation}
and $Q_p: H^1_0(B)\rightarrow \mathbb R$ denotes the quadratic form associated to $L_p$, namely
\[Q_p (v):=\int_B \left[|\nabla v(x)|^2 -p|u_p(x)|^{p-1}v(x)^2  \right]dx.\]
%
%B_p(v,v)\]
%where $B_p:H^1_0(B)\times H^1_0(B)\rightarrow \mathbb R$ is the bilinear form associated to $L_p$ and defined as
%\[B_p(v,w):=\int_B \left[\nabla v \nabla w -p|u_p|^{p-1}vw  \right]dx.
%\]
The {\sl Morse index of $u_p$}, denoted by $m(u_p)$, is the maximal dimension of a subspace $X\subseteq H^1_0(B)$ such that $Q_p(v)<0,  \ \forall v\in X\setminus\{0\}$. Since $B$ is a bounded domain this is equivalent to say that $m(u_p)$ is  the number of the negative eigenvalues of $L_p$ counted with their multiplicity.
\\

Now let $u_p$ be a radial solution to \eqref{problem}, then, if it is sign-changing, from \cite{AftalionPacella} we have the following lower bound on its Morse index which applies in particular to least energy sign-changing radial solutions of \eqref{problem}

\

\begin{lemma}\label{LemaAftalionPacella}
Let $p\in (1, p_S)$ and let $u_p$ be any sign-changing radial solution to \eqref{problem}, then
\[m(u_p)\geq N+2\]
\end{lemma}

\

\begin{proof}
The proof is given in \cite{AftalionPacella} for semilinear equations with general autonomous nonlinearities $f(u)$, showing that the linearized operator $L_p$ has at least $N$ negative eigenvalues whose corresponding eigenfunctions are non-radial and do change sign. Therefore, adding the first eigenvalue, which is obviously associated to a radial eigenfunction, one gets at least $N+1$ negative eigenvalues. In the case when $f$ is superlinear, as for $f(u)=|u|^{p-1}u$, $p>1$, then it is easy to see, testing the quadratic form on the solution $u_p$ in each nodal region, that there are at least as many radial negative eigenfunctions as the number of nodal regions of $u_p$. Therefore $m(u_p)\geq N+2$.
\end{proof}

\

When $u_p$ is a radial solution to \eqref{problem} we can also consider the sequence of the radial eigenvalues of $L_p$
%also define the linear operator \[L_{p,rad}:=L_{p}\arrowvert_{H^1_{0,rad}(B)}
%\]
%where $L_p$ is the linearized operator at $u_p$ defined in \eqref{linearizedOperator}.
%
(i.e. eigenvalues which are associated to a radial eigenfunction) that we denote by  \[\beta_i(p), \quad i\in\N^+\]counting them with their multiplicity.
%
%and  $\psi_{p,i}$ be the corresponding eigenfunctions, namely $\psi_{p,i}(x)=\psi_{p,i}(|x|)=\psi_{p,i}(r)$  which satisfy	 the following ODE
%\[\left\{
%\begin{array}{lr}
%-\psi_{p,i}''-\frac{(N-1)}{r}\psi_{p,i}'-p|u_p|^{p-1}\psi_{p,i}=\beta_i(p) \psi_{p,i} \ \ \ \ r\in (0,1)
%\\
%\\
%\psi_{p,i}'(0)=\psi_{p,i}(1)=0
%\end{array}
%\right.
%\]
%
For the eigenvalues $\beta_i(p)$ an analogous characterization holds:
\begin{eqnarray}\label{CourantCharEigenvRad}
\beta_i(p) &=& \inf_{\substack{
W\subset H^1_{0,rad}(B)\\ dim W=i}}   \max_{\substack{v\in W\\v\neq 0}}\ \ \
R_p[v]
%\nonumber\\
%&=&\inf_{\substack{
%W\subset \mbox{\textcolor{red}{?spazio?}}(0,1)\\ dim W=i}}   \max_{\substack{v\in %W\\v\neq 0}} \frac{\int_0^1\left[(v')^2-p|u_p|^{p-1}v^2\right]r^{N-1} dr}{\int_0^1 %v^2r^{N-1}dr}
\end{eqnarray}
where $R_p$ is as in \eqref{Rayleigh} and $H^1_{0,rad}(B)$ is the subspace of the radial functions of $H^1_0(B)$.
\\

The {\sl radial Morse index of $u_p$}, denoted by $m_{rad}(u_p)$,
%is the maximal dimension of a subspace $X\subseteq H^1_{0,rad}(B)$ such that $Q_p(v)<0, \ \ \forall v\in X\setminus\{0\}$.
%\\
%Again in our case this is equivalent to say that $m_{rad}(u_p)$ is
is then the number of the negative radial eigenvalues $\beta_i(p)$ of $L_{p}$ counted according to their multiplicity.
It is well known (see for instance \cite{BartschWeth}) that for least energy nodal radial solutions $u_p$ to \eqref{problem} we have
\begin{equation}\label{LemmaMorseIndexRadiale}
m_{rad}(u_p)=2
\end{equation}
for any $p\in (1, p_S)$.

 \

\

As mentioned in the introduction, in order to compute the Morse index of $u_p$ we approximate the eigenvalue problem for $L_p$ with analogous problems in annuli. \\

Therefore we consider the annuli
\begin{equation}\label{def:anello}
A_n:=\{x\in \mathbb R^N\ :\ \frac{1}{n}<|x|<1 \}, \quad n\in \mathbb N^+,
\end{equation}
%
%Let $u_p$ be a solution to \eqref{problem} and $L_{p}$ be the operator defined as in \eqref{linearizedOperator}. Let us define \begin{equation}\label{L_p^n} L_{p}^n:=L_{p}\arrowvert_{H^1_{0}(A_n)}\end{equation}
%
%
%
and denote by \[\mu_i^n(p), \quad i\in\N^+\] the Dirichlet eigenvalues of $L_p$ in $A_n$ counted according to their multiplicity.
%
%
% and by $\phi_{p,i}^n$ the corresponding eigenfunctions, namely
%\[
%\left\{
%\begin{array}{lr}
%L_{p}^n\phi_{p,i}^n=\mu_i^n(p)\phi_{p,i}^n \ \ \mbox{ in }A_n\\
%\phi_{p,i}^n = 0\ \mbox{ on }\partial A_n
%\end{array}
%\right.
%\]
Again they can be characterized as
\begin{eqnarray}\label{CourantCharEigenvN}
\mu_i^n(p) &=& \inf_{\substack{
V\subset H^1_{0}(A_n)\\ dim V=i}}   \max_{\substack{v\in V\\v\neq 0}}\ \ \
R_{p}^n[v]
\end{eqnarray}
where $R_{p}^n$ is the corresponding Rayleigh quotient
\begin{equation}\label{RayleighN}
R_p^n[v]:=\frac{Q_p^n(v)}{\int_{A_n} v(x)^2 dx}
\end{equation}
and $Q_p^n: H^1_0(A_n)\rightarrow \mathbb R$ is the associated quadratic form
\[Q_p^n(v):=\int_{A_n}\left(|\nabla v(x)|^2-p|u_p(x)|^{p-1}v(x)^2\right) dx.\]
Let us denote by $k_p^n$ the number of negative eigenvalues $\mu_i^n(p)$.\\

For a  radial solution $u_p$ to \eqref{problem} let us also set
%
%\begin{equation}\label{L_{p,rad}^n}
%L_{p,rad}^n:=L_{p}\arrowvert_{H^1_{0,rad}(A_n)}
%\end{equation}
%and let us denote by
by \[\beta_i^n(p), \quad i\in\N^+\] the radial Dirichlet eigenvalues of $L_p$ in $A_n$ counted with their multiplicity.
Again we have
\begin{eqnarray}\label{CourantCharEigenvNRad}
\beta_i^n (p)&=& \inf_{\substack{
V\subset H^1_{0,rad}(A_n)\\ dim V=i}}   \max_{\substack{v\in V\\v\neq 0}}\ \ \
R_{p}^n[v]
\end{eqnarray}
where $R_p^n$ is as in \eqref{RayleighN}.\\
Finally let $k_{p,rad}^n$ be the number of radial negative eigenvalues of $L_p$ in $A_n$.
\\
\\

It is easy to see, using the canonical embedding $H^1_0(A_n)\subset H^1_0(B)$  and the min-max characterizations \eqref{CourantCharEigenv}, \eqref{CourantCharEigenvN} and \eqref{CourantCharEigenvRad}, \eqref{CourantCharEigenvNRad}, that the following inequalities hold
\begin{equation}\label{ordinamento}
\mu_i^n(p)\geq \mu_i(p)\  \ \ \mbox{ and }\ \ \ \ \beta_i^n(p)\geq \beta_i(p)\ \ \ \ \forall\: i,n\in\mathbb N^+.
\end{equation}
Similarly we have
\begin{equation}\label{monotonia}
\mu_i^n(p)\geq \mu_i^{n+1}(p)\  \ \ \mbox{ and }\ \ \ \ \beta_i^n(p)\geq \beta_i^{n+1}(p)\ \ \ \ \forall \:i,n\in\mathbb N^+.
\end{equation}
By the continuity of the eigenvalues with respect to the domain we have the following:

\

\begin{lemma}\label{lemma:mu_in to mu_i} Let $p\in (1, p_S)$ be fixed. Then
\begin{equation*}\label{eq:convEigenvalues}
\mu_i^n(p) \searrow \mu_i(p)\ \  \ \mbox{ and } \ \ \ \beta_i^n(p) \searrow \beta_i(p)\ \mbox{ as }\ n\rightarrow +\infty\quad\forall\,i\in\N^+.
\end{equation*}
\end{lemma}

\

\begin{proof}
Though the proof relies on standard arguments we write it for the reader's convenience. Let us fix $i\in\mathbb N^+$ and, to shorten the notation, let us drop the dependence on $p$, so we write $
\mu_i^n:=\mu_i^n(p),$ $\mu_i:=\mu_i(p),$   $\beta_i^n:=\beta_i^n(p),$  $\beta_i:=\beta_i(p).$
Moreover for any function $g\in H^1_0(A_n)$ we still denote by $g$ its extension to the whole ball $B$ which is equal to zero in $B\setminus A_n$.\\
By \eqref{ordinamento} it is enough to prove the following
\begin{equation}\label{mainClaim}
\hspace{-0.5pt}\mbox{\textit{Claim}.}\quad
\mbox{For any }\varepsilon >0\mbox{ there exists  }n_{\varepsilon}\in\N^+\mbox{ such that }
\mu_i^n\leq \mu_i +\varepsilon, \mbox{ for }n\geq n_{\varepsilon}
\end{equation}
Let $\varepsilon>0$ be fixed.  Then by the min-max characterization of $\mu_i$ there exists $W_{\varepsilon}\subset H^1_0(B)$, $\dim W_{\varepsilon}=i$ such that
\begin{equation}\label{DefDiInf}
\max_{\substack{w\in W_{\varepsilon}\\w\neq 0}}\
R_{p}[w]< \mu_i +\frac{\varepsilon}{2}
\end{equation}
Let us denote by $w_j^{\varepsilon}$, $j=1,\ldots, i$ an orthogonal basis of $W_{\varepsilon}$, hence $W_{\varepsilon}=span\{w_1^{\varepsilon}, w_2^{\varepsilon},\ldots, w_i^{\varepsilon}\}$ and without loss of generality assume that
$\int_B w_j^{\varepsilon}(x)^2 dx=1$, for any $j=1,\ldots, i$.\\
We point out that for any function $g\in H^1_0(B)$ there exists a sequence $g_n$ compactly supported in $B\setminus\{0\}$ such that $g_n\to g$ in $H^1_0(B)$. It is obviously possible to choose $g_n$ with its support in $A_n$.
Hence there exist sequences $\left(v_{n,j}^{\varepsilon}\right)_n\in H^1_0(A_n)$ such that $v_{n,j}^{\varepsilon}\rightarrow w_j^{\varepsilon}$, for any $j\in\{1,\ldots,i\}$ in $H^1_0(B)$ as $n\rightarrow +\infty$ (extension to zero), $j=1,\ldots, i$.
\\For $n$ large the space $V^{\varepsilon}_n\subset H^1_0(A_n)$,
defined by  \[V^{\varepsilon}_n:=span\{v_{n,1}^{\varepsilon},v_{n,2}^{\varepsilon},\ldots, v_{n,i}^{\varepsilon}\}\]
satisfies
$dim V^{\varepsilon}_n=i$. Indeed if by contradiction there exist $t_{n,j}\in\R$ such that
\[
\sum_{j=1}^i t_{n,j} v_{n,j}^\varepsilon=0\qquad\textnormal{and}\qquad (t_{n,1},\ldots,t_{n,i})\neq(0,\ldots,0)
\]
then also
\begin{equation}\label{comb lin v_nj^eps}\sum_{j=1}^i \frac{t_{n,j}}{\max_j\{|t_{n,j}|\}} v_{n,j}^\varepsilon=0,
\end{equation} but, being bounded, $\frac{t_{n,j}}{\max_j\{|t_{n,j}|\}}\to t_j$, up to subsequences, as $n\to+\infty$ $j=1,\ldots, i$ and it is not difficult to see that, up to a subsequence, there exists $\ell\in\{1,\ldots,i\}$ such that $|t_\ell|=1$. Passing to the limit in \eqref{comb lin v_nj^eps} we get then $\sum_{j=1}^i t_j w_j^\varepsilon=0$ with $|t_\ell|=1$, which is in contradiction with $dim W_{{\varepsilon}}=i$.
\\
We now show the existence of $n_{\varepsilon}\in\N^+$ such that
\begin{equation}\label{miniClaim}
\max_{\substack{v\in V_n^{\varepsilon}\\v\neq 0}}\
R_{p}^n[v]\leq \max_{\substack{w\in W_{\varepsilon}\\w\neq 0}}\
R_{p}[w]+\frac{\varepsilon}{2},\ \ \ \ \mbox{ for } n\geq n_{\varepsilon}
\end{equation}
Since $\mu^n_i\leq\max_{\substack{v\in V_n^{\varepsilon}\\v\neq 0}}R_{p}^n[v]$, \eqref{miniClaim} together with \eqref{DefDiInf} proves \textit{Claim} \eqref{mainClaim} and so the assertion.\\
In order to prove \eqref{miniClaim} we argue by contradiction.
Hence let us assume that there exists a subsequence $n_k\rightarrow +\infty$ such that
\begin{equation}\label{NegazioneMiniClaim}
\max_{\substack{v\in V_{n_k}^{\varepsilon}\\v\neq 0}}\
R_{p}^{n_k}[v]> \max_{\substack{w\in W_{\varepsilon}\\w\neq 0}}\
R_{p}[w]+\frac{\varepsilon}{2},\ \ \ \ \mbox{ for any } k
\end{equation}
Let $\widetilde v_k^{\varepsilon}\in V_{n_k}^{\varepsilon}$, $\widetilde v_k^{\varepsilon}\neq 0$
such that
\[R_p^{n_k}[\widetilde v_k^{\varepsilon}]= \max_{\substack{v\in V_{n_k}^{\varepsilon}\\v\neq 0}}\
R_{p}^{n_k}[v].\]
Since the Rayleigh quotient is $0$-homogeneous  we can assume without loss of generality that
\begin{equation}\label{normalizzo}
\int_{A_{n_k}}\widetilde v_k^{\varepsilon}(x)^2dx=1.
\end{equation}
By definition of the space $V_{n_k}^{\varepsilon}$ there exists $(t_{k,1}^{\varepsilon}, t_{k,2}^{\varepsilon},\ldots, t_{k,i}^{\varepsilon})\in \mathbb R^i$ such that
\[\widetilde v_k^{\varepsilon}=t_{k,1}^{\varepsilon}v_{{n_k},1}^{\varepsilon}\ +\ t_{k,2}^{\varepsilon}v_{{n_k},2}^{\varepsilon}\ + \ \ldots\ +\ t_{k,i}^{\varepsilon}v_{{n_k},i}^{\varepsilon}.\]
Now recalling that each sequence  $v_{{n_k},j}^{\varepsilon}\rightarrow w^{\varepsilon}_j$ in $H^1_0(B)$ as $k\rightarrow + \infty$ for $j=1,\ldots, i$ and that the $w_j^\varepsilon$, $j=1,\ldots,i$, form an orthogonal basis verifying $\|w_j^\varepsilon\|_{L^2(B)}=1$ we deduce that the  sequences $\left(t_{k,j}^{\varepsilon}\right)_k$, $j=1,\ldots, i$ are bounded, being
\begin{eqnarray*}
1&\overset{\eqref{normalizzo}}{=}&\int_{A_{n_k}}\widetilde v_k^{\varepsilon}(x)^2 dx=\sum_{j=1}^i\ \left(t_{k,j}^{\varepsilon}\right)^2\ \int_{A_{n_k}}  v_{{n_k},j}^{\varepsilon}(x)^2 dx\ +o_k(1)\sum_{\underset{j\neq\ell}{j,\ell=1
}}^it_{k,j}^{\varepsilon}\,t_{k,\ell}^{\varepsilon}\\
&=& \ \sum_{j=1}^i\ \left(t_{k,j}^{\varepsilon}\right)^2 \ + \ o_k(1) \ + \ o_k(1)\sum_{\underset{j\neq\ell}{j,\ell=1
}}^it_{k,j}^{\varepsilon}\,t_{k,\ell}^{\varepsilon},
\end{eqnarray*}
then
\[
\sum_{j=1}^i\ \left(t_{k,j}^{\varepsilon}\right)^2\leq 1 \ + \ o_k(1) \ + \ o_k(1)\sum_{j=1}^i\ \left(t_{k,j}^{\varepsilon}\right)^2.
\]
So there exists $t_j^{\varepsilon}\in\mathbb R$ such that up to a subsequence $t_{k,j}^{\varepsilon}\rightarrow t_j^{\varepsilon}\in\mathbb R$, $j=1,\ldots, i$.\\
As  a consequence, passing to  a subsequence, that we continue to  denote by $\left(\widetilde v_k^{\varepsilon}\right)_k$, we get
\[\widetilde v_k^{\varepsilon}\rightarrow w_{\varepsilon}:=t_1^{\varepsilon}w^{\varepsilon}_1\ +\ t_2^{\varepsilon}w^{\varepsilon}_2\ + \ \ldots \ +\ t_i^{\varepsilon}w^{\varepsilon}_i\ \ \ \mbox{ in }\  H^1_{0}(B)\ \mbox{ as }\ k\rightarrow +\infty.\]
Clearly the limit $w_{\varepsilon}\in W_{\varepsilon}$ and moreover
$R_p^{n_k}[\widetilde v_k^{\varepsilon}]=R_p[\widetilde v_k^{\varepsilon}]\rightarrow R_p[w_{\varepsilon}]$  as $k\rightarrow +\infty$.
Passing to the limit in \eqref{NegazioneMiniClaim} as  $k\rightarrow +\infty$ it follows that
\[R_p[w_{\varepsilon}]\geq\max_{\substack{w\in W_{\varepsilon}\\w\neq 0}}\
R_{p}[w]+\frac{\varepsilon}{2},\]
which is a contradiction.\\
In the same way the assertion on the convergence of the radial eigenvalues can be proved.
\end{proof}

\

By Lemma \ref{lemma:mu_in to mu_i} and \eqref{ordinamento} it follows that the number of negative eigenvalues (resp. negative radial eigenvalues) of the linearized operator $L_p$ in $B$  coincides with the number $k_p^n$  (resp. $k_{p,rad}^n$) of negative eigenvalues (resp. negative radial eigenvalues) of $L_p$ in $A_n$, for $n$ large:

\

\begin{lemma} \label{lemma:morseProblemiSenzaPesoAnello}
Let $p\in (1, p_S)$ and let $u_p$ be a solution to \eqref{problem}. Then there exists $ n'_p\in\N^+$ such that:\\
a)  $m(u_p)=k_p^n$ and, if $u_p$ is radial, also $m_{rad}(u_p)= k_{p,rad}^n$ for $n\geq  n'_p$.\\
b) In particular if $u_p$ is the least energy nodal radial solution to \eqref{problem} then by \eqref{LemmaMorseIndexRadiale} it follows that \[k_{p,rad}^n=2\ \mbox{ for }\ n\geq  n'_p.\]
\end{lemma}

\

\

\

\section{Auxiliary weighted eigenvalue problems in annuli}\label{section:auxiliaryweightedoperator}

\

For a radial solution $u_p$ to \eqref{problem}, we consider the following linear operator $\widetilde{L_{p}^n}: H^2(A_n)\cap H^1_0(A_n)\rightarrow L^2(A_n)$:
\begin{equation}
\label{weightedOp}
\widetilde{L_{p}^n} v: = |x|^2 \left(  -\Delta v-p|u_p(x)|^{p-1}v \right), \ \ x\in A_n,
\end{equation}
where $A_n$ are the annuli in \eqref{def:anello} and let us denote by
\[\widetilde{\mu_i^n}(p), \quad i\in\N^+\]
its eigenvalues  counted with their multiplicity. The corresponding eigenfunctions $h_{i,p}^n$ satisfy
\begin{equation}\label{problAutovLpTilde}\left\{
\begin{array}{lr}
-\Delta h_{i,p}^n(x)-p|u_p(x)|^{p-1}h_{i,p}^n(x)=\widetilde{\mu_i^n}(p) \frac{h_{i,p}^n(x)}{|x|^2} \ \ \ \ x\in A_n
\\
\\
h_{i,p}^n=0 \ \ \ \ \mbox{on } \partial A_n
\end{array}
\right.
\end{equation}
Since the singularity $x=0$ does not belong to the annulus $A_n$, the eigenvalues $\widetilde\mu_i^n(p)$ can be characterized as
\begin{equation}\label{CaratVariazAutov}
\widetilde{\mu_i^n}(p)=\inf_{\substack{W\subset H^1_0(A_n)\\dim W=i}}\max_{
\substack{
v\in W\\ v\neq 0}}\frac{\int_{A_n}\left(|\nabla v(x)|^2-p|u_p(x)|^{p-1}v(x)^2\right) dx}{\int_{A_n} \frac{v(x)^2}{|x|^2}dx}
\end{equation}
Let  $\widetilde{ k_{p, }^n}$ be the number of the negative eigenvalues of the operator $
\widetilde{L^n_{p}}$, counted with their multiplicity.\\

Furthermore, since $u_p$ is radial we consider the following linear operator with weight
$\widetilde{L_{p,}^n}_{rad}: H^2((\frac{1}{n},1))\cap H^1_0((\frac{1}{n},1))\rightarrow L^2((\frac{1}{n},1))$
\[\widetilde{ L_{p,}^n}_{rad} v: = r^2\left(  -v''-\frac{(N-1)}{r}v'-p|u_p(r)|^{p-1}v\right), \ \ \ \ r\in (\frac{1}{n},1)\]
and denote by \[\widetilde{\beta_i^n}(p), \quad i\in\N^+\] its eigenvalues  counted with their multiplicity.
Clearly $\widetilde{\beta_i^n}(p)$ is an eigenvalue of $\widetilde{ L_{p,}^n}_{rad}$ if and only if it is a radial eigenvalue of $\widetilde{ L_{p,}^n}$ (i.e. an eigenvalue associated with radial eigenfunctions) and so the following characterization holds true
\begin{eqnarray}\label{defbetatilde1n}
\widetilde{\beta_i^n}(p)&=&
\inf_{\substack{
V\subset H^1_{0,rad}(A_n)\\ dim V=i}}   \max_{\substack{v\in V\\v\neq 0}}\ \ \
\frac{\int_{A_n}\left(|\nabla v(x)|^2-p|u_p(x)|^{p-1}v(x)^2\right) dx}{\int_{A_n} \frac{v(x)^2}{|x|^2}dx}.
%\\
%
%&=&\inf_{
%\substack{
%v\in H^1_0(0,1)\\ v\neq 0}}\frac{\int_{\frac{1}{n}}^1\left[(v')^2-p|u_p|^{p-1}v^2\right]r^{N-1} dr}{\int_{\frac{1}{n}}^1 v^2r^{N-3}dr}.
\end{eqnarray}
Finally by  $\widetilde{ k_{p, }^n}_{rad}$ we mean  the number of negative eigenvalues of the operator $
\widetilde{L^n_{p}}_{rad}$.

\

Denoting by $\sigma(\cdot)$ the spectrum of a linear operator we have the following decomposition result:

\

\begin{lemma}\label{lemma:decompositionOfTheSpectrum}
Let $p\in (1,p_S)$ and  $u_p$ be a radial solution to \eqref{problem}. Then
for any $n\in\mathbb N^+$
\begin{equation}
\sigma(\widetilde{L_p^{n}})
=\sigma(\widetilde{L_{p,}^{n}}_{rad})+\sigma(-\Delta_{S^{N-1}})
\end{equation}
where $\Delta_{S^{N-1}}$ is the Laplace-Beltrami operator on the unit sphere $S^{N-1}$, $N\geq 2$.
\end{lemma}

\begin{proof}
The proof is not difficult, we refer to  \cite{GladialiGrossiPacellaSrikanth} or \cite{BartschClappGrossiPacella}.
\end{proof}

\

By Lemma \ref{lemma:decompositionOfTheSpectrum} we then have that, for any $n\in\mathbb N^+$, the eigenvalues $\widetilde{\mu_j^n}(p)$ of $\widetilde{L_p^{n}}$ are given by
\begin{equation}\label{decomposizioneAutovalori} \widetilde{\mu_j^n}(p)\ =\ \widetilde{\beta_i^n}(p)\ +\ \lambda_k, \ \ \mbox{ for } i,j=1,2,\ldots,\ \ k=0,1, \ldots
\end{equation}
where  $\widetilde{\beta_i^n}(p)$, $i=1,2,\ldots$ are the eigenvalues of the radial operator $\widetilde{L_p^{n}}_{rad}$ and $\lambda_k$, $k=0,1,\ldots$ are the eigenvalues of the Laplace-Beltrami operator $-\Delta_{S^{N-1}}$ on the unit sphere $S^{N-1}$, $N\geq 2$. It is known (\cite[Proposition 4.1]{BerezinShubin}) that
\begin{equation}\label{lambdak}
\lambda_k=k(k+N-2),\ \ k=0,1, \ldots
\end{equation}
with multiplicity
\begin{equation}\label{multiplicity}
N_k-N_{k-2}
\end{equation}
where
\[N_h:=\binom{N-1+h}{N-1}=\frac{(N-1+h)!}{(N-1)!h!},\ \mbox{ if }h\geq 0, \ \ \ N_h=0,\ \mbox{ if }h< 0. \]
It is important to note that in the previous decomposition only the eigenvalues $\widetilde{\beta_i^n}(p)$
depend on the exponent
$p$ while the eigenvalues $\lambda_k$
depend only on the dimension
N.

\

Recall that by the approximation results in Section \ref{section:linearizedOperator} we know that
$m(u_p)=k_p^n$ and $m_{rad}(u_p)= k_{p,rad}^n=2$ for $n$ large, where
$k_p^n$ and  $k_{p,rad}^n$ are, respectively,  the number of  negative eigenvalues  and the number of negative radial eigenvalues of the linearized operator $L_p$ in the annulus $A_n$.
\\

Next result establishes an important equivalence between $k_p^n$ and $k_{p,rad}^n=2$ and the number of negative eigenvalues of the auxiliary weighted operators $\widetilde{L_p^n}$ and $\widetilde{ L_{p,}^n}_{rad}$ that we have introduced in this section:

\

\begin{lemma} \label{lemmaEquivTraPesoESenzaPeso} Let $N\geq 2$, $p\in (1,p_S)$ and $u_p$ be a solution to \eqref{problem}. Then \\
a) the number $ k_{p}^n$ of negative eigenvalues $\mu_i^n(p)$ of $L_p$ in $A_n$ coincides with the number $\widetilde{ k_{p }^n}$ of negative eigenvalues $\widetilde{\mu_i^n}(p)$ of $\widetilde{L_p^n}$;\\
b) if $u_p$ is radial, then the number $k_{p,rad }^n$ of negative radial eigenvalues $\beta_i^n(p)$ of $L_{p}$ in $A_n$ coincides with the number $\widetilde{ k_{p, }^n}_{rad}$ of negative eigenvalues $\widetilde{\beta_i^n}(p)$ of $\widetilde{L_{p,}^n}_{rad}$.
 \end{lemma}

 \

\begin{proof}
The proof of part \emph{a)} is the same as in \cite[Lemma 2.1]{GladialiGrossiPacellaSrikanth} and we repeat it below for completeness, the proof of part \emph{b)} follows similarly, restricting to radial functions.\\
\\
{\sl Step 1.  We show that $k_{p}^n\geq \widetilde{ k_{p }^n}$.}\\
Let $h$ be an eigenfunction for the operator $\widetilde{L_p^n}$ corresponding to a negative eigenvalue $\widetilde{\mu^n}(p) <0$:
\begin{equation}\label{eqAutofunz}
\left\{
\begin{array}{lr}
-\Delta h(x)-p|u_p(x)|^{p-1}h(x)=\widetilde{\mu^n}(p) \frac{h(x)}{|x|^2} \ \ \ \ x\in A_n
\\
h=0 \ \ \ \ \mbox{on } \partial A_n
\end{array}
\right.
\end{equation}
Multiplying \eqref{eqAutofunz} by $h$ and integrating over $A_n$ we get
\[Q_p^n(h)=\int_{A_n} \left[|\nabla h(x)|^2-p|u_p(x)|^{p-1}h(x)^2\right]dx=\widetilde{\mu^n}(p) \int_{A_n}\frac{h(x)^2}{|x|^2}dx <0 \]
namely $h$ makes the quadratic form $Q_p^n$ negative. The conclusion follows from the fact that the set of all these eigenfunctions is a space of dimension $ \widetilde{ k_{p }^n}$.\\
\\
{\sl Step 2. We show that $k_{p}^n\leq \widetilde{ k_{p }^n}$. }\\
Let us assume by contradiction that $k_{p}^n > \widetilde{ k_{p }^n}$ and let $W$ be the $k_p^n$-dimensional space spanned by the orthogonal  eigenfunctions $\varphi_i$ associated to the negative Dirichlet eigenvalues of $L_p$ in $A_n$
\[W:=span\{\varphi_1,\varphi_2,\ldots,\varphi_{k_p^n}\}\ \subset H^1_0(A_n).\]
By the variational characterization \eqref{CaratVariazAutov} of the eigenvalues of $\widetilde{L_p^n}$ we would have
\begin{equation}
\widetilde{\mu_{k_p^n}^n}(p)\leq\max_{
\substack{
v\in W\\ v\neq 0}}\frac{\int_{A_n}\left(|\nabla v(x)|^2-p|u_p(x)|^{p-1}v(x)^2\right) dx}{\int_{A_n} \frac{v(x)^2}{|x|^2}dx}<0,
\end{equation}
 reaching a contradiction.
\end{proof}

\

Combining the previous result with the approximation done in Section \ref{section:linearizedOperator} we get:

\begin{proposition}\label{proposition:MorseRadialeConPeso}
Let $N\geq 2$,  $p\in (1, p_S)$ and $u_p$ be a solution to \eqref{problem}. Then there exists $n'_p\in\N^+$ such that:\\
a) the Morse index $m(u_p)$
 of $u_p$ coincides with the number $\widetilde{k_p^n}$ of negative eigenvalues $\widetilde{\mu_i^n}(p)$ (counted with their multiplicity)   of $\widetilde{L_p^n}$ for $n\geq  n'_p$;\\
b) if $u_p$ is radial, the radial Morse index $m_{rad}(u_p)$
 of $u_p$ coincides with the number $\widetilde{ k_{p, }^n}_{rad}$ of negative eigenvalues $\widetilde{\beta_i^n}(p)$ (counted with their multiplicity) of $\widetilde{L_{p,}^n}_{rad}$ for $n\geq  n'_p$.
\end{proposition}

\

\begin{proof}
It follows with $n_p'\in\N^+$ as in Lemma \ref{lemma:morseProblemiSenzaPesoAnello}, combining the results in Lemma \ref{lemma:morseProblemiSenzaPesoAnello} and Lemma \ref{lemmaEquivTraPesoESenzaPeso}.
\end{proof}

\

\begin{corollary}\label{Cor:hpmorsetradotte}
Let $N\geq 2$, $p\in (1,p_S)$ and $u_p$ be the least energy sign-changing radial solution to \eqref{problem}. Then there exists $ n'_p\in\N^+$ such that:\\
a) $\widetilde{k_p^n}\geq N+2$, for $n\geq  n'_p$;
\\
b) $\widetilde{ k_{p, }^n}_{rad}=2$, for $n\geq  n'_p$.
\end{corollary}

\

\begin{proof}
From Lemma \ref{LemaAftalionPacella},  \eqref{LemmaMorseIndexRadiale} and Proposition \ref{proposition:MorseRadialeConPeso}.
\end{proof}

\

Next result gives an important estimate of the second eigenvalue $\widetilde{\beta_2^n}(p)$ of the auxiliary weighted radial operator $\widetilde{L_{p,}^n}_{rad}$, when $u_p$ is the least energy sign changing radial solution to \eqref{problem}.

\

\begin{proposition}\label{LemmaStimeAutovaloriRadialeConPeso Iparte}
Let $N\geq 2$,  $p\in (1, p_S)$ and $u_p$ be the least energy sign-changing radial solution to \eqref{problem} with $u_p(0)>0$. Then there exists $ n''_p\in\N^+$ such that:
\[\widetilde{\beta_2^n}(p)> -(N-1)\ \ \mbox{ for any }n\geq  n''_p.\]
\end{proposition}

\

\begin{proof}
By Proposition \ref{PropositionUnicoMaxeMin} we now that $u_p$ has $2$ nodal regions and that, letting $r_p\in (0,1)$ be the \emph{nodal radius} as defined in \eqref{rp}, then  $u_p(r)>0$ for $r\in (0,r_p)$, $u_p(r)<0$ for $r\in (r_p,1)$, $\upp(r)$ is strictly decreasing for  $r\in (0,r_p)$ and it has  a unique minimum point  $s_p\in (r_p,1)$. Moreover by the Hopf Lemma $\frac{\partial u_p}{\partial r}(r_p)<0$ and  $\frac{\partial u_p}{\partial r}(1)>0$.
 Let $\eta(r):=\frac{\partial u_p}{\partial r}$. Hence by the above considerations for any $n\geq  n''_p:=[\frac1{r_p}]+1$, $\eta$ satisfies
\[\left\{
\begin{array}{lr}
\widetilde{L_{p,}^n}_{rad}\ \eta=-(N-1)\eta, \quad \ r\in (\frac{1}{n},1)
\\
\eta(\frac{1}{n})<0 \
\\
\eta(1)> 0
\end{array}
\right.
\]
and moreover $\eta$ has a unique zero in the interval $(\frac 1n,1)$ if $n\geq  n''_p$. \\
Let $w$ be an eigenfunction of $\widetilde{L_{p,}^n}_{rad}$ associated with the eigenvalue $\widetilde{\beta_2^n}$, namely
\[\left\{
\begin{array}{lr}
\widetilde{L_{p,}^n}_{rad}\ w=\widetilde{\beta_2^n} w, \quad \ r\in (\frac{1}{n},1)
\\
w(\frac{1}{n})=0
\\
w(1)= 0
\end{array}
\right.
\]
Assume by contradiction that $\widetilde{\beta_2^n}\leq-(N-1)$.\\
If $\widetilde{\beta_2^n}=-(N-1)$ then $\eta$ and $w$ are two solutions of the same Sturm-Liouville equation
\[(r^{N-1}v')'+\left[p|u_p(r)|^{p-1}r^{N-1}+\frac{\widetilde{\beta_2^n}}{r^{3-N}}\right]v=0,\quad \ \ \ r\in (\frac{1}{n},1)\]and they are linearly independent because $\eta(1)\neq 0=w(1)$.
As a consequence (Sturm Separation Theorem) the zeros of $\eta$ and $w$ must alternate. Since $\eta $ has a unique zero in $(\frac{1}{n},1)$, this implies that $w>0$ in $(\frac{1}{n},1)$ and so $\widetilde{\beta_2^n}=\widetilde{\beta_1^n}$.
\\
 If $-(N-1)>\widetilde{\beta_2^n}$ then by the Sturm Comparison Theorem, $\eta$ must have a zero between  any two consecutive zeros of $w$. As a consequence, since $\eta $ has a unique zero, it must be $w>0$ in $(\frac{1}{n},1)$ and again $\widetilde{\beta_2^n}=\widetilde{\beta_1^n}$ which is not possible.
\end{proof}

\

\

\

\section{A limit weighted eigenvalue problem}\label{section:limitWeighted}

\

In this section we consider the weighted operator
\[\widetilde L^* v:=|x|^{2}\left[-\Delta v-V(x)v\right],\qquad x\in\mathbb R^N, \qquad N\geq 2,\]
where $V$ is defined as follows
\begin{equation}
\label{limiteV}
V(x):=
\left\{
\begin{array}{lr}
e^{U(x)}=\left(\frac{1}{1+\frac{1}{8}|x|^2}\right)^2  &   \mbox{ if } N= 2.\\
\\
p_S U^{p_S-1}(x)=\frac{N+2}{N-2}\left( \frac{N(N-2)}{N(N-2)+|x|^2} \right)^{2}
  &   \mbox{ if } N\geq 3
\end{array}
\right.
\end{equation}
and $U$ is defined as in \eqref{U} if $N=2$, while for $N\geq3$
\begin{equation}\label{UNgeq3}
    U(x):=\left( \frac{N(N-2)}{N(N-2)+|x|^2} \right)^{\frac{N-2}{2}}
    \end{equation}
is the unique positive bounded radial solution to the critical equation
\[\left\{
\begin{array}{lr}
-\Delta U=U^{\frac{N+2}{N-2}}\ \mbox{ in }\mathbb R^N\\
U(0)=1.
\end{array}
\right.
\]

\

We are interested in computing the first eigenvalue of $\widetilde L^*$  and  exhibit an associated eigenfunction.
In order to define the first eigenvalue we need first to introduce a suitable space of functions.
Let us recall that $D^{1,2}(\R^N)$ is the Hilbert space defined as the closure of $C^{\infty}_c(\R^N)$ with respect to the Dirichlet norm $\|v\|_{D^{1,2}(\R^N)}:=\left(\int_{\R^N}|\nabla v(x)|^2dx\right)^{\frac{1}{2}}$ and let us denote by $D^{1,2}_{rad}(\R^N)$ the subspace of the radial functions in $D^{1,2}(\R^N)$. Moreover let $L^2_{\frac{1}{|x|}}(\R^N)$ be the Hilbert space
\[L^2_{\frac{1}{|x|}}(\R^N):=\left\{v:\R^N\rightarrow \R \ : \ \frac{v}{|x|}\in L^2(\R^N)    \right\}\]
endowed with the scalar product $(u,v):=\int_{\R^N}\frac{u(x)v(x)}{|x|^2}dx$.
\\Then we can define the space
\begin{equation}\label{Drad}
D_{rad}(\R^N):=D^{1,2}_{rad}(\R^N)\cap L^2_{\frac{1}{|x|}}(\R^N)
\end{equation}
%
%\begin{equation}D_{rad}(\R^N):=\left\{u\in D^{1,2}_{rad}(\mathbb R^N)\ : \ \int_{\mathbb R^N}\frac{u^2}{|x|^2}<+\infty\right\},
%\end{equation}
%
endowed with the scalar product
\[
(u,v)=\int_{\R^N}\nabla u(x)\nabla v(x)\,dx+\int_{\R^N}\frac{u(x)\,v(x)}{|x|^2}dx.
\]
 Observe that
$D_{rad}(\R^N)$  defined in \eqref{Drad} is an Hilbert space and obviously it embeds continuously both in  $D^{1,2}_{rad}(\R^N)$ and in $L^2_{\frac{1}{|x|}}(\R^N)$. Moreover by the Hardy inequality (\cite{Hardy, HardyLittlePolya, OpicKufner}) $D_{rad}(\R^N)= D^{1,2}_{rad}(\mathbb R^N)$ when $N\geq 3$,  while it is well known that $D_{rad}(\R^2)\subsetneq D^{1,2}_{rad}(\mathbb R^2)$. \\

\

Let us set
\begin{eqnarray}\label{defbetastar}
\widetilde\beta^*:=
%&=&\inf_{
%\substack{
%v\in D_{rad}\\ v\neq 0}}\frac{\widetilde R^*(v)}{\int_{\mathbb R^N} \frac{v^2}{|x|^2}dx}
%\nonumber
%\\
%&=&
\inf_{
\substack{
v\in D_{rad}(\R^N)\\ v\neq 0}} \widetilde R^*(v)
\end{eqnarray}
where
\[\widetilde R^*(v):=\frac{\widetilde Q^*(v)}{\|\frac{v}{|x|}\|_{L^2(\R^N)}^2},\]
\[\widetilde Q^*(v):=\int_{\mathbb R^N}\left(|\nabla v(x)|^2-V(x)v(x)^2\right) dx\]
and $D_{rad}(\R^N)$ is the space in \eqref{Drad}.

Since $x\mapsto V(x)|x|^2$ is bounded, $\widetilde Q^*(v)$ and $\widetilde R^*(v)$  are well defined for $v\in D_{rad}(\R^N)$, indeed one has $\int_{\mathbb R^N}|\nabla v(x)|^2dx<\infty$ and $\int_{\R^N}V(x)v(x)^2 dx\leq \sup_{\mathbb R^N}(V(x)|x|^2)\int_{\mathbb R^N}\frac{v(x)^2}{|x|^2}dx=C\int_{\mathbb R^N}\frac{v(x)^2}{|x|^2}dx<\infty$.

\

\

Our main result is the following:
\begin{theorem}\label{lemma:betastar} For any $N\geq 2$
\[\widetilde\beta^*=-(N-1) \]
and it is achieved at the function
\begin{equation}\label{eta1}
\eta_1(x)=\left\{
            \begin{array}{ll}
              \frac{|x|}{1+\frac18|x|^2} & \hbox{if $N=2$} \\
             \frac{|x|}{(1+\frac{|x|^{2}}{N(N-2)})^{\frac{N}{2}}} & \hbox{if $N\geq3$}
            \end{array}
          \right..
\end{equation}
%Moreover the minimizing sequences are compact in $L^2_{\frac{1}{|x|^2}}(\R^N)$.
\end{theorem}

\

The  proof of Theorem \ref{lemma:betastar} is postponed at the end of the section. Here we start with the following:

\

\begin{proposition}\label{unicheSoluzioniDiEquazioneLimite}
Let $\lambda\leq 0$ and let $\eta\in C^2( \R^N\setminus\{0\})\cap D_{rad}(\R^N)$, $\eta\geq0$, $\eta\neq 0$, be a radial solution to
 \begin{equation}\label{equazioneLimite}
 -\lap\eta(x)-V(x)\eta(x)=\lambda\frac{\eta(x)}{|x|^2}\qquad x\in \R^N\setminus\{0\}\\
 \end{equation}
 Then \[\lambda=-(N-1).\]
\end{proposition}

\

\begin{proof}
It is easy to check that the function $\eta_1$ in \eqref{eta1}
is a solution to
\begin{equation}\label{equazionealpha1}
 -\lap\eta_1(x)-V(x)\eta_1(x)=\lambda_1\frac{\eta_1(x)}{|x|^2}\qquad x\in \R^N\setminus\{0\}
\end{equation}
with $\lambda_1=-(N-1)$.
Let us assume that there exists a function $\eta_2\in C^2( \R^N\setminus\{0\})\cap D_{rad}(\R^N)\setminus\{0\}$, radial and nonnegative solving
\begin{equation}\label{equazionealpha2}
  -\lap\eta_2(x)-V(x)\eta_2(x)=\lambda_2\frac{\eta_2(x)}{|x|^2}\qquad x\in \R^N\setminus\{0\}\\
\end{equation}
for some $\lambda_2\leq 0$.

Being $\int_{\R^N}|\nabla\eta_2|^2 dx<+\infty$ there exist two sequences of radii $r_n\to0$ and $R_n\to+\infty$ such that
\begin{equation*}
r_n^N|\nabla\eta_2(r_n)|^2\to0\qquad\mbox{ and} \qquad R_n^N|\nabla\eta_2(R_n)|^2\to0\mbox{ as }n\to+\infty,
\end{equation*}
so in particular
\begin{equation}\label{Raggetti}
r_n^N|\nabla\eta_2(r_n)|\to0\qquad\textrm{and}\qquad|\nabla\eta_2(R_n)|\to0\quad\textrm{as $n\to+\infty$}.
\end{equation}
Besides, applying Lemma \ref{lemma:appendixN3} and Lemma \ref{lemma:appendixLinfinito} in the Appendix we get that
\begin{equation}\label{dalemmiappendix}
r_n^{N-1}\eta_2(r_n)\to0\qquad\textrm{and}\qquad\frac{\eta_2(R_n)}{R_n}\to0\quad\textrm{as $n\to+\infty$}.
\end{equation}
Next, multiplying \eqref{equazionealpha1} by $\eta_2$ and \eqref{equazionealpha2} by $-\eta_1$, adding them and integrating over $B_{R_n}(0)\setminus B_{r_n}(0)$ we get
\begin{eqnarray}\label{ortogonalita}
(\lambda_1-\lambda_2)\int_{B_{R_n}(0)\setminus B_{r_n}(0)}\frac{\eta_1(x)\eta_2(x)}{|x|^2}dx&=&\underbrace{\int_{\partial B_{R_n}(0)}\eta_1\nabla\eta_2\cdot\nu\  dS}_{=:A_n}+\underbrace{\int_{\partial B_{R_n}(0)}\eta_2\nabla\eta_1\cdot\nu\  dS}_{=:B_n}\nonumber\\
& &-\underbrace{\int_{\partial B_{r_n}(0)}\eta_1\nabla\eta_2\cdot\nu\  dS}_{=:C_n}-\underbrace{\int_{\partial B_{r_n}(0)}\eta_2\nabla\eta_1\cdot\nu\  dS}_{=:D_n}
\end{eqnarray}
where $\nu$ is the outer normal to $\partial B_{R_n}(0)$.
Then by virtue of the previous considerations and using the explicit expression  of $\eta_1$ in \eqref{eta1}, we can estimate $A_n$, $B_n$ $C_n$ and $D_n$ as follows:
\begin{eqnarray*}
|A_n|&\leq & c_N R_n^{N-1}\eta_1(R_n)|\nabla\eta_2(R_n)|\leq c_N R_n^{N-1} R_n^{-(N-1)}|\nabla\eta_2(R_n)|\stackrel{\eqref{Raggetti}}{\longrightarrow}0\qquad\mbox{ as $n\to+\infty$},
\\
|B_n|&\leq & c_N R_n^{N-1}\eta_2(R_n)|\nabla\eta_1(R_n)|\leq c_N R_n^{N-1}\eta_2(R_n) R_n^{-N}=\frac{c_N\eta_2(R_n)}{R_n}\stackrel{\eqref{dalemmiappendix}}{\longrightarrow}0\qquad\mbox{ as $n\to+\infty$},
\\
|C_n|&\leq & c_N r_n^{N-1}\eta_1(r_n)|\nabla\eta_2(r_n)|\leq c_N r_n^{N-1} r_n|\nabla\eta_2(r_n)|\stackrel{\eqref{Raggetti}}{\longrightarrow}0\qquad\mbox{ as $n\to+\infty$},
\\
|D_n|&\leq& c_N r_n^{N-1}\eta_2(r_n)|\nabla\eta_1(r_n)|\leq c_N r_n^{N-1}\eta_2(r_n) \stackrel{\eqref{dalemmiappendix}}{\longrightarrow}0\qquad\mbox{ as $n\to+\infty$},
\end{eqnarray*}
where in the above estimates we have denoted by $c_N$ a generic constant depending only on $N$.
\\
Thus passing to the limit in \eqref{ortogonalita} we get
\[
(\lambda_1-\lambda_2)\int_{\R^N}\frac{\eta_1(x)\eta_2(x)}{|x|^2}dx=0
\]
which implies that $\lambda_2=\lambda_1=-(N-1)$, because $\eta_1>0$ in $\R^N\setminus\{0\}$ and by assumption $\eta_2\geq0$, $\eta_2\not\equiv 0$.
\end{proof}

\

\begin{lemma}\label{lemma:betastarnegativo} $\widetilde\beta^*\leq -(N-1)\ (<0)$.
\end{lemma}

\

\begin{proof}
Let $\eta_1$ be the function defined in \eqref{eta1}. Then $\eta_1\in D_{rad}(\R^N)$ and satisfies the equation \eqref{equazioneLimite}. Multiplying  it
by $\eta_1$ and integrating over $\R^N$ we get
\[\widetilde R^*(\eta_1)= -(N-1)\]
and the conclusion follows recalling the definition of $\widetilde\beta^*$  in \eqref{defbetastar}.
\end{proof}

\

\

\begin{proof}[Proof of Theorem \ref{lemma:betastar}]
First we show a coercivity property:
%there exist $C_1, C_2>0$ such that  then
%\begin{equation}
%\label{coercivita}
%\widetilde R^*(v)=\int_{\mathbb R^N}\left(|\nabla v|^2-V(x)v^2\right) dx \geq C_1\int_{\mathbb R^N}|\nabla v|^2dx- C_2\
%\end{equation}
%for all
% $v\in D_{rad}(\R^N)$, $\int_{\mathbb R^N} \frac{v^2}{|x|^2}dx=1$.\\
%This clearly implies that $\widetilde\beta^*>-\infty$.\\
%\\
%Proof of \eqref{coercivita}
for all
$v\in D_{rad}(\R^N)$, $\int_{\mathbb R^N} \frac{v(x)^2}{|x|^2}dx=1$
\begin{eqnarray}\label{coercivita}
\widetilde Q^*(v)& = &  \int_{\R^N} |\nabla v(x)|^2dx- \int_{\R^N} V(x)|x|^2\frac{v(x)^2}{|x|^2}dx
\nonumber\\
&\geq & \int_{\R^N} |\nabla v(x)|^2dx -\sup_{\R^N} \left(V(x)|x|^2\right)\int_{\R^N}\frac{v(x)^2}{|x|^2} dx\nonumber\\
&=& \int_{\R^N} |\nabla v(x)|^2dx - C
\end{eqnarray}
where $(0<)\ C:=\sup_{\R^N} \left(V(x)|x|^2\right)<\infty$.
Since one can easily show that
\[\widetilde\beta^* = \ \inf_{
\substack{
v\in D_{rad}(\R^N)\\ \|\frac{v}{|x|}\|_{L^2(\R^N)}^2 =1}} \widetilde Q^*(v),\]
then clearly \eqref{coercivita} implies that $\widetilde\beta^*>-\infty$.\\
%% PARTE VECCHIA
%\textcolor{blue}{ Let $R>0$ then
%\begin{eqnarray*}
% \int_{\mathbb R^N}V(x)v^2 &=& \int_{|x|\leq R}V(x)v^2 dx+ \int_{|x|>R}V(x)v^2\\
% &\leq& C\int_{|x|\leq R}v^2+\int_{|x|>R}V(x)v^2
% \\
% &\leq& R^2 C \int_{|x|\leq R}\frac{v^2}{|x|^2}+ \int_{|x|>R}V(x)v^2
% \\
%&\leq & R^2 C \int_{\R^N}\frac{v^2}{|x|^2}+ \int_{|x|>R}V(x)v^2\\
%&= &R^2 C + \int_{|x|>R}V(x)v^2\qquad \mbox{(by Strauss in $D^{1,2}$ \textcolor{red}{(only for $N\geq 3$)})}
%\\
%&\leq& R^2 C + C_N\|\nabla v\|^2_{L^2(\R^N)}\int_{|x|>R}V(x)|x|^{-(N-2)}
%\\
%&\leq & R^2 C + C_N\|\nabla v\|^2_{L^2(\R^N)}\int_{|x|>R}\frac{1}{|x|^{4+(N-2)}}\\
%&=& R^2 C + C_N\|\nabla v\|^2_{L^2(\R^N)}\frac{1}{4 R^2}
%\end{eqnarray*}}
%
%Hence
%
%
%\begin{eqnarray*}
%\widetilde R^*(v)&=&\int_{\mathbb R^N}\left(|\nabla v|^2-V(x)v^2\right) dx\\
%& \geq&  \int_{\R^N}|\nabla v|^2\left(1 -\frac{C_N}{4R^2} \right)-R^2C
%\end{eqnarray*}
%and  the conclusion follows by choosing $C_1:=\left(1 -\frac{C_N}{4R^2} \right)$ and $C_2:=R^2C$ with  $R>0$ fixed and big enough.\\
%\textcolor{red}{NB: la norma $L^2$ puo' esplodere per questo ho dovuto usare strauss in $D^{1,2}_{rad}$, ma vale solo per $N\geq 3$}
%
%
Let $(v_n)_n\subset D_{rad}(\R^N)$ be  a minimizing sequence for \eqref{defbetastar} with $\| \frac{v_n}{|x|}  \|_{L^2(\R^N)}=1$.
Clearly we can assume without loss of generality that $v_n\geq0$ (because otherwise we could consider $|v_n|$).
By the coercivity property \eqref{coercivita} it follows that $v_n$ is bounded in $D^{1,2}_{rad}(\R^N)$ and hence in $D_{rad}(\R^N)$, being $\| \frac{v_n}{|x|}  \|_{L^2(\R^N)}=1$. Therefore, by the reflexivity of $D_{rad}(\R^N)$, there exists $v\in D_{rad}(\R^N)$ such that up to a subsequence
 \begin{eqnarray*}
 && v_n\rightharpoonup v \qquad \mbox{ in  } D_{rad}(\R^N)\\
 &&v_n\rightarrow v \qquad \mbox{ in  } L^{q}(B_R), \  1<q<+\infty \mbox{ if }N=2;\ 1<q<\frac{2N}{N-2} \mbox{ if }N\geq 3\\
 && v_n\rightharpoonup v \qquad \mbox{ in  } D^{1,2}_{rad}(\R^N)\qquad \mbox{ by the continuous embedding of $D_{rad}(\R^N)$ into $D^{1,2}_{rad}(\R^N)$  }\\
 && v_n\rightharpoonup v \qquad \mbox{ in  } L^2_{\frac{1}{|x|^2}}(\R^N)\qquad \mbox{ by the continuous embedding of $D_{rad}(\R^N)$ into $L^2_{\frac{1}{|x|^2}}(\R^N)$  }
\\
&& v_n\rightarrow v\qquad a.e. \mbox{ in }\R^N.
 \end{eqnarray*}
Hence  $v\geq 0$,
\begin{equation}\label{lsc}
\|\nabla v\|_{L^2(\R^N)}\leq \liminf_{n\to+\infty}\|\nabla v_n\|_{L^2(\R^N)}
\end{equation}
and
\begin{equation}\label{vleq1}
\left\| \frac{v}{|x|}\right\|_{L^2(\R^N)}\leq\liminf_{n\to+\infty}\left\| \frac{v_n}{|x|}\right\|_{L^2(\R^N)}=1.
\end{equation}
 Next we show that
 \begin{equation} \label{ContinSecondTerm}
 \int_{\R^N}V(x) v_n(x)^2dx\rightarrow \int_{\R^N}V(x) v(x)^2dx \ \ \mbox{ as }\ n\rightarrow +\infty.
 \end{equation}
Let us fix $\varepsilon >0$ then
\begin{eqnarray*}
\left|\int_{\{|x|>R\}} V(x) (v_n(x)^2-v(x)^2)dx\right|&\leq &\sup_{|x|>R}(V(x)|x|^2)\left[\int_{\{|x|>R\}} \frac{v_n(x)^2}{|x|^2}dx+\int_{\{|x|>R\}} \frac{v(x)^2}{|x|^2}dx\right]\\
&\stackrel{\eqref{vleq1}}{ \leq }&\frac{C}{R^2}<\frac{\varepsilon}{2},
\end{eqnarray*}
choosing $R$ sufficiently large.\\
On the other hand, fixing the same $R$, since $v_n\to v$ in $L^2(B_R)$, also
\[
V^{\frac12}v_n\to V^{\frac12}v \quad \mbox{in $L^2(B_R)$}
\]
and hence
\[
\int_{B_R} V(x) v_n(x)^2\,dx\to\int_{B_R} V(x) v(x)^2 dx.
\]
Therefore  for $n$ large
\[
\left|\int_{B_{ R}} V(x) v_n(x)^2-\int_{B_{ R}}V(x) v(x)^2\right|<\frac{\epsilon}{2},
\]
thus proving \eqref{ContinSecondTerm}.

%
%
% Indeed by using Strauss inequality again for any $\varepsilon>0$
% \begin{eqnarray*}
%\left| \int_{\R^N}V(x) \left(v_n^2-v^2\right)dx  \right| &\leq &
%\left| \int_{|x|\leq R}V(x) \left(v_n^2-v^2\right)dx  \right|+  \int_{|x|>R}V(x) v_n^2 dx +  \int_{|x|>R}V(x) v^2 dx\\
%&\leq & V(0)\left| \int_{|x|\leq R} \left(v_n^2-v^2\right)dx  \right|\ +\\
%&&\ + \ C_N
%\left(\|\nabla v_n\|^2_{L^2(\R^N)}+\|\nabla v\|^2_{L^2(\R^N)}\right)\int_{|x|>R}\frac{V(x)}{|x|^{N-2}}dx\\
%&\leq & V(0)\left| \int_{|x|\leq R} \left(v_n^2-v^2\right)dx  \right| \ + \ 2C_N\|\nabla v_n\|^2_{L^2(\R^N)}\int_{|x|>R}\frac{1}{|x|^{N+2}}dx\\
%&\leq &   V(0)\left| \int_{|x|\leq R} \left(v_n^2-v^2\right)dx  \right| \ + \ \frac{C}{R^2}\leq \varepsilon
% \end{eqnarray*}
% for $R>0$ sufficiently large fixed and $n$ large.
% \\
 By \eqref{lsc}, \eqref{ContinSecondTerm} and Lemma \ref{lemma:betastarnegativo} it follows that
 \begin{eqnarray} \label{Rstarneg}
\widetilde Q^*(v)&=& \int_{\mathbb R^N}\left(|\nabla v(x)|^2-V(x)v(x)^2\right) dx\leq \liminf_n  \int_{\mathbb R^N}\left(|\nabla v_n(x)|^2-V(x)v_n(x)^2\right) dx\nonumber\\
&=&\widetilde\beta^*\leq -(N-1)<0,
\end{eqnarray}
in particular $\widetilde Q^*(v)<0$ and so $v\neq 0$.
\\
Next we show that
\begin{equation}\label{NomaPesataUno}
 \left\|\frac{v}{|x|}\right\|_{L^2(\R^N)}=1.
 \end{equation}
By the  definition of $\widetilde\beta^*$
 and \eqref{Rstarneg} we have
 \begin{equation}\label{lei}
 \widetilde\beta^*\leq \widetilde R^*(v)=\frac{\widetilde Q^*(v)}{\|\frac{v}{|x|}\|_{L^2(\R^N)}^2}\leq \frac{\widetilde\beta^*}{\|\frac{v}{|x|}\|_{L^2(\R^N)}^2}.
 \end{equation}
 Since $\widetilde\beta^*<0$ then necessarily
 \[
 \left\|\frac{v}{|x|}\right\|_{L^2(\R^N)}\geq 1
 \]
 which together with \eqref{vleq1} gives \eqref{NomaPesataUno}.
 As a consequence from \eqref{lei} we get
 \[\widetilde R^*(v)=\widetilde\beta^*\] namely the infimum of problem \eqref{defbetastar} is attained at $v$.
 \\
Finally since $v\geq0$, $v\neq0$, is a radial solution to
\[
-\lap v(x)-V(x) v(x)=\widetilde\beta^*\frac{v(x)}{|x|^2}\qquad x\in\R^N
\]
with $\widetilde\beta^*<0$ we can apply Proposition \ref{unicheSoluzioniDiEquazioneLimite} obtaining that $\widetilde\beta^*=-(N-1)$.
\end{proof}

%
%
%
%
%
%
% \begin{eqnarray}
%\int_{\left(\frac{A_n}{\varepsilon_p}\cap\{|x|> R\}\right)}V(x)\widehat{\phi_{p,n}}^2 dx &\leq &\int_{\{|x|>R\}}V(x) \widehat{\phi_{p,n}}^2 dx\\
%&\leq & C_N\|\nabla \widehat{\phi_{p,n}}\|^2_{L^2(\R^N)} \int_{\{|x|>R\}}V(x) \frac{1}{|x|^{N-2}}dx\\
%&\leq & C_N\|\nabla \widehat{\phi_{p,n}}\|^2_{L^2(\R^N)} \int_{\{|x|>R\}} \frac{C}{|x|^{N+2}}dx\\
%&= & C_N\|\nabla \widehat{\phi_{p,n}}\|^2_{L^2(\R^N)}  \frac{C}{R^{2}}dx\\\ \ \  \sup_{|x|>R}\left( \frac{1}{1+\frac{|x|^2}{8}} \right)^{2}\leq   \frac{64}{R^4}
%\end{eqnarray}
%
%
%

\

\

\

\section{$N=2$: asymptotic analysis of the eigenvalues $\widetilde{\beta_1^n}(p)$}\label{section:analysis}

\

In this section we focus on the case $N=2$ and we study the value of the first eigenvalue $\widetilde{\beta_1^n}(p)$ of the auxiliary weighted radial operator $\widetilde{L_{p,}^n}_{rad}$, when $u_p$ is the least energy sign changing radial solution to \eqref{problem}.

\

Our results concern the asymptotic behavior  as $p\rightarrow +\infty$ of a family of eigenvalues
\begin{equation}\label{np}
\widetilde\beta_1(p):=\widetilde\beta_1^{n_p}(p)\qquad\qquad \mbox{ with }\  n_p:=\max\{ n'_p, n''_p,[(\ep_p^+)^{-2}]+1\}
\end{equation}
where $n'_p$ is defined in Corollary \ref{Cor:hpmorsetradotte}, while $n''_p$ is introduced in Proposition \ref{LemmaStimeAutovaloriRadialeConPeso Iparte} and $\varepsilon^+_p$ is defined in \eqref{epsilon+-}.\\
 Notice that this choice of $n_p$ and Corollary \ref{Cor:hpmorsetradotte} imply that $\widetilde\beta_1(p)<0$ for every $p>1$.

 \

The main result of this section is the following.

\

\begin{theorem}\label{thm:betaplimiteN=2} Let $N=2$, then
\[
\lim_{p\to+\infty}\widetilde{\beta_1}(p)= -\frac{\ell^2+2}{2}\simeq -26.9,
\]
where $\ell$ is defined as in \eqref{varie}.
\end{theorem}

\

We emphasize that while all the results in the previous sections hold true in any dimension $N\geq 2$ and for any $p\in (1, p_S)$,  Theorem \ref{thm:betaplimiteN=2} is related only to the case $N=2$ and $p\rightarrow +\infty$. Indeed, as we will see, the proof relies on the precise  asymptotic behavior as $p\rightarrow +\infty$ of  $u_p$ when $N=2$, which has been investigated in \cite{GrossiGrumiauPacella2, DeMachisIanniPacellaJEMS} as already recalled in Section \ref{section:preliminaries}.

\

For any fixed $p>1$ let us set
\begin{equation}
\label{definition:Ap}
A_p:=A_{n_p}=\{y\in \mathbb R^2\ :\ \frac{1}{n_p}<|y|<1 \}
\end{equation}
%and also
%\begin{eqnarray}
%\label{definition:App}
%&&  A_p^{+}:=\{y\in \mathbb R^2\ :\ \frac{1}{n_p}<|y|< r_p \} \\
%&&  A_p^{-}:=\{y\in \mathbb R^2\ :\ r_p\leq |y|< 1\} \label{definition:Apm}
%\end{eqnarray}
%where $r_p$ is the nodal radius of $u_p$.
%
and  let $\phi_p$ be the (radial and positive) eigenfunction of $\widetilde{L_{p,}^{n_p}}_{rad}$ associated with the first eigenvalue $\widetilde\beta_1(p)$, which satisfies, for $r=|x|$
\begin{equation}
\label{primaautofunz}
\left\{
\begin{array}{lr}
-\phi_{p}''-\frac{\phi_{p}'}{r}-p|u_p|^{p-1}\phi_{p}=\widetilde{\beta}_1(p) \frac{\phi_{p}}{r^2}, \ \ \ \ r\in (\frac{1}{n_p},1)
\\
\phi_{p}(\frac{1}{n_p})=\phi_{p}(1)=0,
\end{array}
\right.
\end{equation}
and normalized in such a way that
\begin{equation}\label{normalizzazione}
\left\|\frac{\phi_{p}}{|y|}\right\|_{L^2(A_p)}=1.
\end{equation}

\

\begin{lemma}\label{lemma:boundGradiente}
There exists $C>0$ such that
\[
\sup\{\|\nabla\phi_p\|^2_{L^2(A_p)}\,:\,p\in(1,+\infty)\}\leq C.
\]
\end{lemma}

\

\begin{proof}
Since $\widetilde\beta_1(p)<0$ and recalling that  $p|\upp(y)|^{p-1}|y|^2\leq C$ for any $y\in B$ (see \eqref{Q3}) we have:
\begin{eqnarray*}
\int_{A_p}|\nabla\phi_p(y)|^2dy&=&\int_{A_p}p|\upp(y)|^{p-1}\phi_p(y)^2 dy+\widetilde\beta_1(p)\int_{A_p}\frac{\phi_p(y)^2}{|y|^2}dy\\
&\leq&\int_{A_p}p|\upp(y)|^{p-1}|y|^2\frac{\phi_p(y)^2}{|y|^2} dy\leq C\int_{A_p}\frac{\phi_p(y)^2}{|y|^2}dy=C
\end{eqnarray*}
where the last equality follows by \eqref{normalizzazione}.
\end{proof}

\

We start by deriving a, still inaccurate, estimate from below of $\widetilde{\beta}_1(p)$ that will be useful in the sequel.

\begin{lemma}\label{lemma:beta1pbounded}
There exists $C>0$ such that
\begin{equation}\label{betaUnoLimitato}
-C\leq \widetilde{\beta}_1(p)\ (< 0).
\end{equation}
\end{lemma}
\begin{proof}

By \eqref{primaautofunz}, multiplying by $\phi_p$ and integrating over $A_p$ we have
\begin{eqnarray*}
\int_{A_p}|\nabla\phi_p(y)|^2dy&=&\int_{A_p}p|\upp(y)|^{p-1}\phi_p(y)^2 dy+\widetilde\beta_1(p)\int_{A_p}\frac{\phi_p(y)^2}{|y|^2}dy\\
&=&\int_{A_p}\left(p|\upp(y)|^{p-1}|y|^2+\widetilde\beta_1(p)\right)\frac{\phi_p(y)^2}{|y|^2} dy\\
&\leq &  \max_{y\in B} \left(p|\upp(y)|^{p-1}|y|^2\right)+ \widetilde\beta_1(p),
\end{eqnarray*}
 where we have used \eqref{normalizzazione}.
 As a consequence $\widetilde{\beta_1}(p)\geq -\max_{y\in B} \left( p|\upp(y)|^{p-1}|y|^2\right)\geq -C$,
 where  the last inequality follows from \eqref{Q3}.
%
%ALTRA PROOF:
%
%Since $\phi_p$ is positive and $\phi_p(\frac1{n_p})=\phi_p(1)=0$, $\phi_p$ has at least a maximum point $m_p$ and in that point by \eqref{primaautofunz}
%\[
%\left(\frac{\widetilde{\beta_1}(p)}{m_p^2}+p|\upp(m_p)|^{p-1}\right)\phi_p(m_p)\geq0
%\]
%so $-\widetilde{\beta_1}(p)\leq p|\upp(m_p)|^{p-1} m_p^2\leq C$ where the last inequality follows from \eqref{Q3}.
\end{proof}

Next we give a bound from above of $\widetilde{\beta}_1(p)$, for $p$ large.

\begin{lemma}\label{lemma:primastimabetapN=2}
We have
\[
\limsup_{p\rightarrow +\infty} \widetilde{\beta_1}(p)\leq-\frac{\ell^2+2}{2}.
\]
\end{lemma}

\

\begin{proof}
We want to show that for any $\varepsilon>0$ there exists $p_{\varepsilon}>1$ such that  for any $p\geq p_\ep$
\begin{equation}\label{primastimabetapN=2}
\widetilde{\beta_1}(p)\leq-\frac{\ell^2+2}{2}+\ep.
\end{equation}
The claim follows considering the radial function $\Psi_{R,p}:\overline{B}\to[0,+\infty)$
\begin{equation}\label{funzionitestN=2}
\Psi_{R,p}(y):=\left\{
              \begin{array}{ll}
                \frac{\psi_p(\frac{\delta\ep^-_p}{R})}{(\frac{\delta\ep^-_p}{R})}(|y|-\frac{\delta\ep^-_p}{R}) & \hbox{$|y|\in[\frac{\delta\ep^-_p}{2R},\frac{\delta\ep^-_p}{R})$} \\
                \psi_p(|y|) & \hbox{$|y|\in[\frac{\delta\ep^-_p}{R},R\delta\ep^-_p]$} \\
                -\frac{\psi_p(R\delta\ep^-_p)}{(R\delta\ep^-_p)}(|y|-2R\delta\ep^-_p) & \hbox{$|y|\in[R\delta\ep^-_p,2R\delta\ep^-_p)$} \\
                0 & \hbox{$|y|\in[0,\frac{\delta\ep^-_p}{2R}]\cup[2R\delta\ep^-_p,1]$}
              \end{array}
            \right.
\end{equation}
for $R$ sufficiently large, where $\psi_p:[0,1]\to[0,+\infty)$ is defined as follows
\begin{equation}\label{psiN=2}
\psi_p(r):=\frac{(\frac{r}{\delta\ep^-_p})^{\frac{2+\gamma}{2}}}{1+(\frac{r}{\delta\ep^-_p})^{2+\gamma}},
\end{equation}
for $\delta$ as in \eqref{varie}.
Indeed, for $p$ large enough, being $\Psi_{R,p}\in H^1_{0,rad}(A_p)$, by the variational characterization of $\widetilde{\beta_1}(p)$  in \eqref{defbetatilde1n} and Lemma \ref{lemma:appendixN=2} in the Appendix we get
\begin{eqnarray}\label{primaBISstimabetapN=2}
\nonumber \widetilde{\beta_1}(p)& \stackrel{\eqref{defbetatilde1n} }{\leq} &\frac{\int_{A_p}|\nabla\Psi_{R,p}(y)|^2-p|\upp(y)|^{p-1}\Psi_{R,p}(y)^2dy}{\int_{A_p}\frac{\Psi_{R,p}(y)^2}{|y|^2}dy}
\\
&\stackrel{\mbox{\scriptsize Lemma \ref{lemma:appendixN=2}} }{\leq} &-\frac{\ell^2+2}{2}(1+o_R(1)+o_p(1))\overset{\eqref{varie}}{\approx} -26.9\ (1+o_R(1)+o_p(1)).
\end{eqnarray}
Note that the function $\Psi_{R,p}$ that we use to evaluate $\widetilde{\beta_1}(p)$ is obtained
by suitably cutting and scaling  $\eta_1$, the eigenfunction associated to the first eigenvalue of the limit weighted operator $\widetilde L^*$ studied in Section \ref{section:limitWeighted} (see \eqref{eta1}), more precisely $\psi_p(r)=\eta_1\left( 2\sqrt{2}(\frac{r}{\delta \varepsilon_p^-} )^{\frac{2+\gamma}{2}} \right)$.
\end{proof}

\

In order to prove Theorem \ref{thm:betaplimiteN=2} one would like to pass to the limit as $p\rightarrow +\infty$ into the equation \eqref{primaautofunz} and deduce the value of $\lim_p \widetilde{\beta}_1(p)$ by studying the limit equation. Anyway since the term $p|u_p|^{p-1}$ is not bounded it is more convenient to consider one of the two  scalings of $\phi_p$, defined for $x\in\frac{A_p}{\varepsilon_p^{\pm}}$, by
\begin{equation}\label{fiCappuccio}
\widehat{\phi_{p}^{\pm}}(x):=\phi_{p}(\varepsilon_p^{\pm} x).
\end{equation}
and pass to the limit in the equation satisfied by it, which is, by \eqref{primaautofunz},
\begin{equation}\label{eq cappuccio pN=2}
\left\{
\begin{array}{lr}
-\Delta\widehat{\phi_{p}^{\pm}}(x)-V^{\pm}_p(x)\widehat{\phi_{p}^{\pm}}(x)=\widetilde{\beta}_1(p) \frac{\widehat{\phi_{p}^{\pm}}(x)}{|x|^2}, \ \ \ \ x\in \frac{A_p}{\varepsilon^{\pm}}%(\frac{1}{n_p\varepsilon_p^{{\pm}}},\frac1{\varepsilon_p^{{\pm}}})
\\
\widehat{\phi_{p}^{\pm}}=0\ \ \ \mbox{ on }\ \partial \left(\frac{A_p}{\varepsilon^{\pm}}\right)%\widehat{\phi_{p}^{\pm}}(\frac{1}{n_p\varepsilon_p^{{\pm}}})=\widehat{\phi_{p}^{\pm}}(\frac1{\varepsilon_p^{{\pm}}})=0
\end{array}
\right.
\end{equation}
where
\begin{equation}
\label{defVp}
V_p^{+}(x)
:=\left|\frac{u_p(\varepsilon_p^{+} x)}{u_p(0)}\right|^{p-1}, \qquad V_p^{-}(x)
:=\left|\frac{u_p(\varepsilon_p^{-} x)}{u_p(s_p)}\right|^{p-1}.
\end{equation}
It is worth to point out that, by definition of $\ep_p^\pm$, by \eqref{epsilonpmN2}, by \eqref{np} and by \eqref{quindi:rapportoepsilon+-} (which implies $\frac{\ep_p^+}{\ep_p^-}\to0$) we have that $\varepsilon_p^{\pm}\to 0$, while $n_p\varepsilon_p^{\pm}\to+\infty$ and so
\begin{equation}\label{npepsilonp}
\frac{A_p}{\varepsilon_p^{\pm}}\to\R^2\setminus\{0\} \qquad\mbox{ as $p\to+\infty$.}
\end{equation}
Moreover $V_p^{\pm}$ is bounded and more precisely, since  by Theorem \ref{thm:GGP2} we have as $p\to+\infty$
\[\begin{array}{lr}
z_p^+\longrightarrow U\quad\mbox{in $C^1_{loc}(\R^2)$\phantom{$\setminus\{0\}_\ell a$} }
\\
z_p^-\longrightarrow Z_\ell\quad\mbox{in $C^1_{loc}(\R^2\setminus\{0\})$}
\end{array}
\]
with $z_p^+$ and $z_p^-$ defined as in \eqref{zp} and \eqref{zm} and $U$ and $Z_\ell$ as in \eqref{U} and \eqref{Zell} respectively,  it follows that, as $p\to+\infty$:
\begin{eqnarray}\label{V^+_ptoV^+}
&& V_p^+=\left|1+\frac{z_p^+}{p}\right|^{p-1}
\rightarrow V^+:=e^{U}
 \ \mbox{ in }\ C^0_{loc}(\mathbb R^2)
\\
&& \label{V^-_ptoV^-}
V_p^-=\left|1+\frac{z_p^-}{p}\right|^{p-1}
\rightarrow V^-:=e^{Z_\ell}
 \ \mbox{ in }\ C^0_{loc}(\R^2\setminus\{0\})\
\end{eqnarray}
Also, denoting still by $\widehat{\phi_{p}^{\pm}}$ the extension to $0$ of $\widehat{\phi_{p}^{\pm}}$ outside of $\frac{A_p}{\ep^{\pm}_p}$, we have that $\widehat{\phi_{p}^{\pm}}$ is bounded in $D_{rad}(\R^2)$, indeed:
%
%\textcolor{red}{
%IDEA: passare al limite nella \eqref{eq cappuccio pnN>3} e dire che l'unica soluzione possibile nell'equazione limite e' per l'autovalore $-(N-1)$.\\
%\\
%LIMITATEZZA DI $\widetilde{\beta_1^n}(p)$:\\
%e' ok, vedi Lemma \ref{lemma:betaplimitato} sotto
%\\
%\\
%LIMITATEZZA DI  $\widehat{\phi_{p,n}}$:\\
%sembra fattibile (in dim 2 e' ok, in dim 3, manca una disuguaglianza per le tower per ora) vedi \eqref{normaPesata1} e Lemma \ref{lemma:boundGradiente} qui di seguito...\\
%\\
%PROBLEMA: la limitatezza di  $\widehat{\phi_{p,n}}$ non ci basta per passare al limite nell'equazione pero', perche' il limite debole potrebbe essere la funzione nulla (la massa scappa ad infinito) e quindi non possiamo concludere niente...\\
%\\
%per evitare questo bisogna avere una qualche compattezza: grazie ai risultati della sezione precedente, ci basterebbe dimostrare che la successione e' una successione minimizzante per il problema agli autovalori limite definito in \eqref{defbetastar}.
%}

\

\begin{lemma}\label{lemma:boundrescalatefi}
There exists $C>0$ such that
\begin{equation} \label{boundGradienteRiscalate}
\sup \{ \| \nabla \widehat{\phi_{p}^\pm}\|_{L^2(\R^2)}\,:\,p\in(1,+\infty)\}\leq C.
\end{equation}
Moreover
\begin{equation}\label{normaPesata1}
\left\| \frac{\widehat{\phi_{p}^{\pm}}}{|x|}   \right\|_{L^2(\R^2)}=1.
\end{equation}
\end{lemma}

\

\begin{proof}
The proof of \eqref{boundGradienteRiscalate}  follows immediately from  the definitions of
$\widehat{\phi_{p}^{\pm}}$, observing that
$
\nabla \widehat{\phi_{p}^{\pm}}(x) = \varepsilon_p^{\pm}\nabla\phi_p(\varepsilon_p^{\pm} x)
$
from which
\begin{equation}\label{bbbb}
\int_{\R^2} |\nabla \widehat{\phi_{p}^{\pm}}(x)|^2 dx =
\int_{\frac{A_p}{\varepsilon_p^\pm}} (\varepsilon_p^\pm)^{2}|\nabla\phi_p(\varepsilon_p^\pm x)|^2 dx=\int_{A_p} |\nabla\phi_p(x)|^2 dx\leq C
\end{equation}
by the bound of $\phi_p$ in
Lemma \ref{lemma:boundGradiente}.\\
\\
 The proof of \eqref{normaPesata1} follows immediately from the definitions \eqref{fiCappuccio}, indeed
%\[\int_{\R^2} \frac{\widehat{\phi_{p}^+}(x)^2}{|x|^2}dx = (\varepsilon_p^{+})^2 \int_{\frac{A_p}{\varepsilon_p^+}} \frac{\phi_{p}(\varepsilon_p^+ x)^2}{|\varepsilon_p^+x|^2} dx=\int_{A_p} \frac{\phi_{p}(y)^2}{|y|^2} dy\overset{\eqref{normalizzazione}}{=}1,\]
%\[\int_{\R^2} \frac{\widehat{\phi_{p}^-}(x)^2}{|x|^2}dx = (\varepsilon_p^{-})^2 \int_{\frac{A_p}{\varepsilon_p^-}} \frac{\phi_{p}(\varepsilon_p^- x)^2}{|\varepsilon_p^-x|^2} dx=\int_{A_p} \frac{\phi_{p}(y)^2}{|y|^2} dy\overset{\eqref{normalizzazione}}{=}1.\]
\[\int_{\R^2} \frac{\widehat{\phi_{p}^\pm}(x)^2}{|x|^2}dx = (\varepsilon_p^{\pm})^2 \int_{\frac{A_p}{\varepsilon_p^\pm}} \frac{\phi_{p}(\varepsilon_p^\pm x)^2}{|\varepsilon_p^\pm x|^2} dx=\int_{A_p} \frac{\phi_{p}(y)^2}{|y|^2} dy\overset{\eqref{normalizzazione}}{=}1.\]
\end{proof}

\

By the results in Lemma \ref{lemma:beta1pbounded} and Lemma \ref{lemma:primastimabetapN=2} and thanks  to \eqref{npepsilonp}, \eqref{V^+_ptoV^+}, \eqref{V^-_ptoV^-}  and Lemma \ref{lemma:boundrescalatefi} we are now in the position to pass to the limit in \eqref{eq cappuccio pN=2}. However the functions $\widehat{\phi_{p}^{\pm}}$ could a priori vanish  and this would not give any limit equation, so the crucial point is to show that actually
$\widehat{\phi_{p}^{-}}$ does not vanish in the limit as $p\rightarrow +\infty$. This will be obtained as consequence of the following nontrivial result:

\

\begin{proposition}\label{prop:m>0}
There exists $K>1$ such that
\[
\liminf\limits_{p\to+\infty}\int_{\{|x|\in[\frac{1}{K},K]\}}\frac{{\widehat{\phi_{p}^{-}}(x)}^2}{|x|^2} dx>0.
\]
\end{proposition}

\

The proof of Proposition \ref{prop:m>0} needs several ingredients: the results of Section \ref{section:limitWeighted}, the definition of $\widehat{\phi_{p}^{\pm}}$ and its properties, the convergence result in \eqref{V^+_ptoV^+},  Lemma \ref{lemma:primastimabetapN=2}. Moreover it strongly depends on the asymptotic behavior of $u_p$ in dimension $N=2$, in particular we need to analyze the behavior of the function $f_p(r):=p |u_p(r)|^{p-1}r^2$ in the positive and the negative nodal region of $u_p$, which is done next and  leads to Proposition
\ref{proposition:bound_fp_+} and Proposition \ref{proposition:bound_fp_-} below. The proof of Proposition \ref{prop:m>0} is therefore postponed after the study of $f_p$.
\\
Finally the conclusion of the proof of Theorem \ref{thm:betaplimiteN=2}, obtained passing to the limit in the equation of $\widehat{\phi_{p}^{-}}$,  is postponed at the end of the section. As it will be clear from the proof, the great part of the contribution to the limit in Theorem \ref{thm:betaplimiteN=2} comes from the negative nodal region of $u_p$.

\

\

\subsection{Study of the function $f_p(r)=p|\upp(r)|^{p-1}r^2$}

\

\

We aim now to study the behavior of the function
\begin{equation}\label{f_p}
f_p(r)=p|\upp(r)|^{p-1}r^2\qquad\qquad\mbox{for $r\in[0,1]$.}
\end{equation}
where $u_p$ is the least energy nodal radial solution to \eqref{problem} when $N=2$.

\

\begin{lemma}\label{lemma:unico_pto_max_fp_in+} The function $f_p$ has a unique critical point $c_p$, which is a point of maximum, in $(0,r_p)$, where $r_p$ is the nodal radius of $\upp$ as in \eqref{rp}. Moreover $f_p$ is strictly increasing for $r\in(0,c_p)$ and strictly decreasing for $r\in(c_p,r_p)$.
\end{lemma}

\

\begin{proof}
Since, for $r\in(0,r_p)$, $\upp(r)$ is nonnegative and
\[
f'_p(r)=p(\upp(r))^{p-2}r [(p-1)\upp'(r)r+2\upp(r) ],
\]
we have that $c_p\in(0,r_p)$ is a critical point of $f_p$ if and only if
\begin{equation}\label{upp'(c_p)}
-\upp'(c_p)=\frac{2\upp(c_p)}{(p-1) c_p}.
\end{equation}
Let $c_p\in(0,r_p)$ be a critical point of $f_p$. Then computing the seconde derivative of $f_p$ we get
\[
f''_p(c_p)=p(\upp(c_p))^{p-2}c_p[(p-1)\upp''(c_p) c_p+(p+1)\upp'(c_p)],
\]
thus $f''_p(c_p)$ has the same sign of
\begin{eqnarray*}
(p-1)\upp''(c_p) c_p+(p+1)\upp'(c_p)&\stackrel{\eqref{problem}}{=}&(p-1)(-\frac{\upp'(c_p)}{c_p}-(\upp(c_p))^p) c_p+(p+1)\upp'(c_p)\\
&\stackrel{\eqref{upp'(c_p)}}{=}& \frac{2\upp(c_p)}{c_p}-(\upp(c_p))^p(p-1) c_p-\frac{p+1}{p-1}\frac{2\upp(c_p)}{c_p}\\
&=&\frac{\upp(c_p)}{c_p}[-(p-1) c_p^2(\upp(c_p))^{p-1}-\frac4{p-1}]<0
\end{eqnarray*}
and therefore $c_p$ is a strict maximum point. Being $f_p(0)=f_p(r_p)=0$ and $f_p>0$ for any $r\in(0,r_p)$ the assertion follows immediately. Indeed note that there cannot be two points of maxima otherwise there should be a minimum point in between .
\end{proof}

\

\begin{proposition}\label{proposition:bound_fp_+}
For any $\ep>0$ there exists $p_\ep >1$ such that for any $p\geq p_\ep$:
\[
f_p(c_p)=\max_{r\in[0,r_p]}f_p(r)\leq 2+\ep,
\]
with $r_p$ as in \eqref{rp}.
\end{proposition}

\

\begin{proof}
We set, for $s\in[0,\frac{r_p}{\ep^+_p})$, $g_p(s):=f_p(\ep^+_p s)$. Then, by definition of $\ep^+_p$ (see \eqref{epsilon+-}), and \eqref{V^+_ptoV^+} we obtain:
\begin{equation}\label{gptog}
g_p(s)=V_p^+(s)s^2\underset{p\rightarrow +\infty}{\longrightarrow} V^+(s)s^2=\left(\frac{s}{1+\frac{s^2}{8}}\right)^2=:g(s)\quad\mbox{in $C^0_{loc}([0,+\infty))$}.
\end{equation}
%
%
%
%
%
%%%
%%%%%
%%%%%%%%
%%%%%%%             un po' oscuro e nebuloso.... ho riscritto sotto questa parte .....
%%%%%%
%%%
%%%
%%
%     Then by Lemma \ref{lemma:unico_pto_max_fp_in+}, by \eqref{gptog} and since $g$ in $[0,+\infty)$ is increasing up to $s=\sqrt 8$, decreasing %     after $s=\sqrt 8$, with $g(\sqrt 8)=2$ we can conclude that $\max_{r\in[0,r_p]}f_p(r)=f_p(c_p)\leq 2+o_p(1)$.
%
%
%
%
%
%
Observe that for the function $g$ it holds: $g>0$ in $(0,\infty)$, $g(0)=0$,  $g(s)\rightarrow 0$ as $s\rightarrow +\infty$,  it has a unique strict maximum at  $s=\sqrt{8}$ with $g(\sqrt{8})=2$ and it is strictly increasing for $s<\sqrt{8}$ and strictly decreasing for $s>\sqrt{8}$.
\\
Let $\varepsilon>0$ and let $K_{\varepsilon}>\sqrt{8}$ be sufficiently large so that $g(K_{\varepsilon})\leq \varepsilon$, then by \eqref{gptog}
\begin{equation}\label{convg}
g_p\rightarrow g\ \mbox{ in }[0,K_{\varepsilon}]\quad\mbox{uniformly.}
\end{equation}
Hence in particular there exists $p_{\varepsilon}>1$ such that for $p\geq p_{\varepsilon}$
\begin{eqnarray}
&& f_p(0)=g_p(0)\leq g(0)+\varepsilon=\varepsilon \label{condizioni1}\\
&& f_p(\varepsilon_p^+ \sqrt{8})=g_p(\sqrt{8})\geq g(\sqrt{8})-\varepsilon= 2-\varepsilon \label{condizioni2}\\
&& f_p(\varepsilon_p^+ K_{\varepsilon})=g_p(K_{\varepsilon})\leq g(K_{\varepsilon})+\varepsilon\leq 2 \varepsilon \label{condizioni3}
\\
&& f_p(r)=g_p(\frac{r}{\varepsilon_p^+})\leq  g(\sqrt{8})+\varepsilon=2+\varepsilon\qquad\forall r\in [0,\varepsilon_p^+ K_{\varepsilon}] \label{condizioni4}
\end{eqnarray}
but $[0,\varepsilon_p^+ K_{\varepsilon}]\subset [0,r_p]$ for $p$ sufficiently large (since $\frac{r_p}{\varepsilon_p^+}\rightarrow +\infty$ by \eqref{quindi:rapportoepsilon+-}) and by Lemma \ref{lemma:unico_pto_max_fp_in+} we know that in $[0,r_p]$ the function $f_p$ has a unique maximum point $c_p$ and that it is strictly increasing for $r<c_p$ and strictly decreasing for $r>c_p$. Thus  \eqref{condizioni1}-\eqref{condizioni2}-\eqref{condizioni3} necessarily imply that for $p$ large $c_p\in (0,\varepsilon_p^+ K_{\varepsilon})$.
The conclusion then follows by \eqref{condizioni4} applied at $r=c_p$.
\end{proof}

\

\begin{lemma}\label{lemma:unico_pto_max_fp_in-}
The function $f_p$ has a unique critical point $d_p$, which is a point of maximum, in $(r_p,1)$, where $r_p$ is the nodal radius of $\upp$ defined in \eqref{rp}. Moreover $f_p$ is strictly increasing for $r\in(r_p,d_p)$ and strictly decreasing for $r\in(d_p,1)$.
\end{lemma}

\

\begin{proof}
Exactly as in the proof of Lemma \ref{lemma:unico_pto_max_fp_in+} we have that $d_p\in(r_p,1)$ is a critical point of $f_p$ if and only if
\begin{equation}\label{upp'(d_p)}
\upp'(d_p)=\frac{2|\upp(d_p)|}{(p-1) d_p}.
\end{equation}
Moreover, for any critical point $d_p\in(r_p,1)$,
$$f''_p(d_p)=p|\upp(d_p)|^{p-1}[-(p-1)|\upp(d_p)d_p^2-\frac{4}{p-1}]<0,$$
Therefore, since $f_p(r_p)=f_p(1)=0$ and $f_p>0$ for any $r\in(r_p,1)$ the assertion follows immediately.
\end{proof}

\

\begin{proposition}\label{proposition:bound_fp_-}
There exists $K>1$ and $p_{ K}>1$ such that for any $p\geq p_{ K}$:
\[
\max_{r\in[r_p,\frac{\ep_p^-}{ K}]\cup[\ep_p^-K,1]}f_p(r)\leq 2,
\]
with $r_p$ as in \eqref{rp} and $\varepsilon_p^-$ as in \eqref{epsilon+-}.
\end{proposition}

\

\begin{proof}
We set, for $s\in(\frac{r_p}{\ep^-_p},\frac{1}{\ep^-_p})$, $h_p(s):=f_p(\ep^-_p s)$. Then, by definition of $\ep^-_p$, and \eqref{V^-_ptoV^-} we obtain:
\begin{equation}\label{hptoh}
h_p(s)=V_p^-(s)s^2\underset{p\rightarrow +\infty}{\longrightarrow} V^-(s)s^2=\frac{2(\gamma+2)^2\delta^{\gamma+2}s^{\gamma+2}}{(\delta^{\gamma+2}+s^{\gamma+2})^2}=:h(s)\quad\mbox{in $C^0_{loc}((0,+\infty))$},
\end{equation}
where the positive constants $\gamma$ and $\delta$ are as in \eqref{varie}.\\
Observe that for the function $h$ it holds: $h>0$ in $(0,\infty)$, $h(s)\rightarrow 0$ as  $s\to0^+$,  $h(s)\rightarrow 0$ as $s\rightarrow +\infty$, it has a unique strict maximum at  $s=\delta$ with $h(\delta)=\ell^2+2 >51$ (see \eqref{varie} for the definition and the value of $\ell$) and it is strictly increasing for $s<\delta$ and strictly decreasing for $s>\delta$.\\
Hence there exists $K>0$ sufficiently large such that $\frac{1}{K}<\delta<K$ and $h(s)\leq1$ for any $s\in(0,\frac{1}{ K}]\cup[K,+\infty)$. Moreover, by \eqref{hptoh}
\begin{equation}\label{convggg}
h_p\rightarrow h\ \mbox{ in }[\frac{1}{K},K]\quad\mbox{uniformly.}
\end{equation}
hence in particular there exists $p_{K}>1$ such that for $p\geq p_{K}$
\begin{eqnarray}
&& f_p(\frac{\varepsilon_p^-}{K})=h_p(\frac{1}{K})\leq h(\frac{1}{ K})+1\leq 2 \label{condizioni11}\\
&& f_p(\varepsilon_p^- \delta)=h_p(\delta)\geq h(\delta)-1= \ell^2+1 >50 \label{condizioni22}\\
&& f_p(\varepsilon_p^-K)=h_p(K)\leq h(K)+1\leq 2  \label{condizioni33}
%\\
%&& f_p(r)=g_p(\frac{r}{\varepsilon_p^+})\leq  g(\sqrt{8})+\varepsilon=2+\varepsilon\qquad\forall r\in [0,\varepsilon_p^+ K] \label{condizioni44}
\end{eqnarray}
But $[\frac{\varepsilon_p^-}{K},\varepsilon_p^-K]\subset [r_p,1]$ for $p$ sufficiently large (since $\varepsilon_p^-\rightarrow 0$ and $\frac{r_p}{\varepsilon_p^-}\rightarrow 0$ by \eqref{quindi:rapportoepsilon+-}) and by Lemma \ref{lemma:unico_pto_max_fp_in-} we know that in $[r_p,1]$ the function $f_p$ has a unique maximum point $d_p$ and that it is strictly increasing for $r<d_p$ and strictly decreasing for $r>d_p$. Hence  \eqref{condizioni11}-\eqref{condizioni22}-\eqref{condizioni33} necessarily imply that for $p$ large $d_p\in (\frac{\varepsilon_p^-}{K},\varepsilon_p^-K)$ and
\[
\begin{array}{lr}
f_p(r)\leq f_p(\frac{\varepsilon_p^-}{K}) \stackrel{\eqref{condizioni11}}\leq 2\qquad \mbox{ for }\ r\in [r_p,\frac{\varepsilon_p^-}{ K}]\\
f_p(r)\leq f_p(\varepsilon_p^-K)\stackrel{\eqref{condizioni33}}\leq 2 \qquad \mbox{ for }\ r\in [\varepsilon_p^-K, 1]
\end{array}
\]
from which the conclusion follows.
%
%
%
%   spiegato un po di piu
%
%
%
%   Then by Lemma \ref{lemma:unico_pto_max_fp_in-}, by \eqref{hptoh} and since $h$ in $(0,+\infty)$ is increasing up to $s=\gamma$ and decreasing %   after $s=\gamma$ we can conclude that $\frac{d_p}{\ep^-_p}\to\delta$.
%   It is easy to see that $h(s)\to0$ as $s\to0^+$ and as $s\to+\infty$, so there exists $\widetilde K>0$ sufficiently large such that $\frac{1} % {\widetilde K}<\delta<\widetilde K$ and $h(s)\leq1$ for any $s\in(0,\frac{1}{\widetilde K}]\cup[\widetilde K,+\infty)$. Then since by \eqref{hptoh} $h_p(\frac{1}{\widetilde K})\to h(\frac{1}{\widetilde K})\leq1$, $h_p(\widetilde K)\to h(\widetilde K)\leq1$ and for $p$ large %enough $h_p$ is decreasing for $s\leq\frac1{\widetilde K}$ and decreasing for $s\geq \widetilde K$ the thesis follows.
%
%
%
%
\end{proof}

\

\subsection{Proof of Proposition \ref{prop:m>0}}

\begin{proof}[Proof of Proposition \ref{prop:m>0}]
By Theorem \ref{lemma:betastar} we know that the value $-1$ coincides with the first radial eigenvalue $\widetilde{\beta^*}$ of the limit weighted operator $\widetilde L^*$. Hence  by evaluating the Rayleigh quotient related to the variational characterization \eqref{defbetastar} of  $\widetilde{\beta^*}$
on the functions $\widehat{\phi_p^+}$ defined in \eqref{fiCappuccio} we get
\begin{eqnarray}\label{stimadalbassoparzialeN=2}
-1 \stackrel{\mbox{\scriptsize{Theorem \ref{lemma:betastar}}}}{=} \widetilde\beta^*
&\stackrel{\eqref{defbetastar}}{\leq} &  \int_{\mathbb R^N}\left(|\nabla \widehat{\phi_{p}^+}(x)|^2-V^+(x)\widehat{\phi_{p}^+}(x)^2\right) dx
\nonumber
\\
&\stackrel{\eqref{eq cappuccio pN=2}+\eqref{normaPesata1}}{=}&\widetilde\beta_1(p) + \int_{\frac{A_p}{\varepsilon_p^+}}\left[V_p^+(x)-V^+(x)\right]\widehat{\phi_{p}^+}(x)^2 dx
\end{eqnarray}
where the set $A_p$ is defined in \eqref{definition:Ap}, $V_p^+$ in \eqref{defVp} and and $V^+=e^{U}$ by \eqref{V^+_ptoV^+}.
Next we estimate the term  $\int_{\frac{A_p}{\varepsilon_p^+}}\left[V_p^+(x)-V^+(x)\right]\widehat{\phi_{p}^+}(x)^2 dx$.\\
Let $\varepsilon\in(0,\frac13)$ and let us fix $R \geq\frac{8}{\sqrt{\varepsilon}}$:
\begin{eqnarray}\label{interm}
\int_{\frac{A_p}{\varepsilon_p^+}}\left[V_p^+(x)-V^+(x)\right]\widehat{\phi_{p}^+}(x)^2 dx
&\leq &
\int_{\frac{A_p}{\varepsilon_p^+}\cap\{|x|\leq R\}}\left|V_p^+(x)-V^+(x)\right|\widehat{\phi_{p}^+}(x)^2 dx
+
\int_{\frac{A_p}{\varepsilon_p^+}\cap\{|x|> R\}}V^+(x)\widehat{\phi_{p}^+}(x)^2 dx
\nonumber \\
&& + \int_{\frac{A_p}{\varepsilon_p^+}\cap\{|x|> R\}}V_p^+(x)\widehat{\phi_{p}^+}(x)^2 dx
\nonumber\\
&=&
I_p \ + \ II_p\ + \ III_p.
\end{eqnarray}
For the term $I_p$ we may use the convergence result in  \eqref{V^+_ptoV^+}, so there exists $p_R>1$ such that for any $p\geq p_R$
\begin{eqnarray*}
I_p &= & \int_{\frac{A_p}{\varepsilon_p^+}\cap\{|x|\leq R\}} \left|V_p^+(x)-V^+(x)\right| \widehat{\phi_{p}^+}(x)^2 dx
\leq \sup_{B_R(0)}\left|V_p^+(x)-V^+(x)  \right|R^2\int_{\mathbb R^2}
\frac{\widehat{\phi_{p}^+}(x)^2}{|x|^2} dx
\\
&\stackrel{\eqref{normaPesata1}}{=} &  \sup_{B_R(0)}\left|V_p^+(x)-V^+(x)  \right|R^2
\stackrel{\eqref{V^+_ptoV^+}}{\leq }\varepsilon.
\end{eqnarray*}
Moreover for any $p>1$ and by our choice of $R$:
\begin{eqnarray*}
II_p &=& \int_{\frac{A_p}{\varepsilon_p^+}\cap\{|x|> R\}} e^{U(x)}|x|^2\frac{\widehat{\phi_{p}^+}(x)^2}{|x|^2} dx
 \leq \sup_{|x|>R} \left(e^{U(x)}|x|^2 \right)\int_{\frac{A_p}{\varepsilon_p^+}\cap\{|x|> R\}} \frac{\widehat{\phi_{p}^+}(x)^2}{|x|^2} dx
 \\
 &\stackrel{\eqref{U}}{\leq} & \frac{64}{R^2} \int_{\R^2}\frac{\widehat{\phi_{p}^+}(x)^2}{|x|^2} dx\stackrel{\eqref{normaPesata1}}{=}\frac{64}{R^2} \leq \varepsilon.
\end{eqnarray*}
We turn now to the estimate of $III_p$ for which we will need Proposition \ref{proposition:bound_fp_+} and Proposition \ref{proposition:bound_fp_-}. Let $K>1$ be as in Proposition \ref{proposition:bound_fp_-}.
First observe that  by \eqref{quindi:rapportoepsilon+-} there exists $p_{R,K}>1$ such that
\begin{equation}\label{ancheRegioPosi}
R\varepsilon_p^+<r_p< \frac{\varepsilon_p^-}{K} \ \mbox{ for } \ p\geq  p_{R,K}.
\end{equation}
Thus for  $p\geq\max\{p_{R,K}, p_{\epsilon}, p_K\}$ where $p_{\varepsilon}$ is as in Proposition \ref{proposition:bound_fp_+}
 and $p_K$ is as in Proposition \ref{proposition:bound_fp_-}, we have
\begin{eqnarray*}
III_p &=& \int_{\frac{A_p}{\varepsilon_p^+}\cap\{|x|> R\}} V_p^+(x)\widehat{\phi_{p}^+}^2(x) dx\\
&=&\int_{A_p\cap\{ |y|>\varepsilon_p^+ R\}}p|u_p(y)|^{p-1}|y|^2\frac{\phi_p^2(y)}{|y|^2}dy
\\
&\stackrel{\eqref{f_p}}{=} &   \int_{A_p \cap\{ |y|>\varepsilon_p^+ R\}}f_p(|y|)\frac{\phi_{p}(y)^2}{|y|^2} dy
\\
&\stackrel{\eqref{ancheRegioPosi}}{=} &  \int_{\{R\varepsilon_p^+\leq |y|\leq \frac{\varepsilon_p^-}{K}\}\cup \{K \varepsilon_p^- \leq |y|\leq 1\}}
f_p(|y|)\frac{\phi_{p}(y)^2}{|y|^2} dy\\
&&\qquad\qquad\qquad\qquad
+
\int_{\{\frac{\varepsilon_p^-}{K}\leq |y|\leq K \varepsilon_p^-\}}f_p(|y|)\frac{\phi_{p}(y)^2}{|y|^2} dy
\\
&\stackrel{\eqref{normalizzazione}}{\leq} &
\max_{\{R\varepsilon_p^+\leq r\leq \frac{\varepsilon_p^-}{K}\}\cup \{K \varepsilon_p^- \leq r\leq 1\} }
 f_p(r)
\ +\
\int_{\{\frac{\varepsilon_p^-}{K}\leq |y|\leq K \varepsilon_p^-\}}f_p(|y|) \frac{\phi_{p}(y)^2}{|y|^2} dy
\\
&\stackrel{\scriptsize \begin{array}{cr} \mbox{ Propositions \ref{proposition:bound_fp_+}-\ref{proposition:bound_fp_-}}\\ \eqref{Q3}\end{array}}{\leq} & 2
\ +\ \varepsilon \ +\ C\int_{\{\frac{\varepsilon_p^-}{K}\leq |y|\leq K \varepsilon_p^-\}}  \frac{\phi_{p}(y)^2}{|y|^2} dy.
\end{eqnarray*}
Finally combining the estimates of $I_p, II_p$ and $III_p$  with \eqref{interm} we get, for $p$ sufficiently large,
\[
\int_{\frac{A_p}{\varepsilon_p^+}}\left[V_p^+(x)-V^+(x)\right]\widehat{\phi_{p}^+}(x)^2 dx \leq  2+3\ep+C\int_{\{|y|\in[\frac{\ep^-_p}{K},K\ep^-_p]\}}\frac{{\phi_{p}}(y)^2}{|y|^2} dy
\]
and so by \eqref{stimadalbassoparzialeN=2} we get that for $p$ sufficiently large
\begin{equation}\label{IVp}
\widetilde\beta_1(p)\geq  -3- 3\ep -C\int_{\{|y|\in[\frac{\ep^-_p}{K},K\ep^-_p]\}}\frac{\phi_{p}(y)^2}{|y|^2} dy\stackrel{\varepsilon<\frac13}{\geq}-4-C\int_{\{|x|\in[\frac{1}{K},K]\}}\frac{{\widehat{\phi_{p}^-}(x)}^2}{|x|^2} dx,
\end{equation}
where for the last inequality we also did a change of variable.
Now, if by contradiction
\[\liminf\limits_{p\to+\infty}\int_{\{|x|\in[\frac{1}{K},K]\}}\frac{{\widehat{\phi_{p}^-}(x)}^2}{|x|^2} dx=0,\] then by \eqref{IVp} we would get
$\limsup\limits_{p\to+\infty}\widetilde\beta_1(p)>-4$ which is impossible by Lemma \ref{lemma:primastimabetapN=2}.
\end{proof}

\

\subsection{Proof of Theorem \ref{thm:betaplimiteN=2}}

\

\

We are finally ready to prove Theorem \ref{thm:betaplimiteN=2}.
\begin{proof}[Proof of Theorem \ref{thm:betaplimiteN=2}]
Let us consider the scaled functions $\widehat{\phi_{p}^-}$ defined in \eqref{fiCappuccio}.
For any fixed $\rho\in C^{\infty}_0(\R^2\setminus\{0\})$ we have for $p$ sufficiently large that $supp \rho\subset A_p$ and so by \eqref{eq cappuccio pN=2}
\begin{equation}\label{equaziocappuccio}
\int_{\R^2 \setminus\{0\}}\nabla \widehat{\phi_{{p}}^-}(x)\nabla\rho(x)\, dx
-
\int_{\R^2 \setminus\{0\}}V_{p}^-(x)\widehat{\phi_{{p}}^-}(x)\rho(x)\, dx
-
\widetilde{\beta_1}({p})\int_{\R^2 \setminus\{0\}}\frac{\widehat{\phi_{{p}}^-}(x)\rho(x)}{|x|^2}\, dx=0.
\end{equation}
We want to pass to the limit as $p\rightarrow +\infty$ into \eqref{equaziocappuccio}.
By Lemma \ref{lemma:boundrescalatefi} we know that $\widehat{\phi_{p}^-}$ is bounded in $D_{rad}(\R^2)$, hence there exists $\widehat{\phi}\in D_{rad}(\R^2)$ such that up to a subsequence
\[\widehat{\phi_{p}^-}\rightharpoonup \widehat{\phi} \qquad \mbox{ in  } D_{rad}(\R^2) \quad\mbox{ as $p\rightarrow +\infty$}\]
and so by the continuous embedding of $D_{rad}(\R^2)$ into $D^{1,2}_{rad}(\R^2)$ and $L^2_{\frac{1}{|x|^2}}(\R^2)$ respectively also
\begin{eqnarray}
&&\widehat{\phi_{p}^-}\rightharpoonup \widehat{\phi} \qquad \mbox{ in  } D^{1,2}_{rad}(\R^2)\label{weakD}\\
 && \widehat{\phi_{p}^-}\rightharpoonup \widehat{\phi} \qquad \mbox{ in  } L^2_{\frac{1}{|x|^2}}(\R^2).\label{weakPeso}
 \end{eqnarray}
Moreover for any bounded set $M\subset \R^2$, by the compact embedding $H^1(M)\subset L^2(M)$ we have
\begin{equation}\label{M}\widehat{\phi_{p}^-}\rightarrow \widehat{\phi} \qquad \mbox{ in  } L^{2}(M)
\end{equation}
and so also
\begin{equation}\label{limiteae} \widehat{\phi_{p}^-}\rightarrow \widehat{\phi}\qquad a.e. \mbox{ in }\R^2.
\end{equation}
Observe that by \eqref{limiteae} $\widehat{\phi}\geq 0$.
Next we show that
\begin{equation}\label{finonnulla}\widehat{\phi}\not\equiv 0.
\end{equation}
Indeed by Proposition \ref{prop:m>0} there exists $K>1$ such that
\begin{equation}\label{BellaCondizioneN=2}
\liminf\limits_{p\to+\infty}\int_{\{|x|\in[\frac{1}{K},K]\}}\frac{\widehat{\phi_{p}^-}(x)^2}{|x|^2} dx
%=\liminf\limits_{p\to+\infty}\int_{\{|y|\in[\frac{\ep^-_p}{K},K\ep^-_p]\}}\frac{{\phi_{p}}(x)^2}{|y|^2} dy
=:m>0.
\end{equation}
Hence taking $M=\{|x|\in[\frac{1}{K},K]\}$, by \eqref{M} we have
\[\int_{\{|x|\in[\frac{1}{K},K]\}}\frac{\widehat{\phi_{p}^-}(x)^2}{|x|^2} dx\leq K^2  \int_{\{|x|\in[\frac{1}{K},K]\}}\widehat{\phi_{p}^-}(x)^2 dx \longrightarrow K^2  \int_{\{|x|\in[\frac{1}{K},K]\}}\widehat{\phi}(x)^2 dx\quad\mbox{ as $p\rightarrow +\infty$}\]
and so combining this with \eqref{BellaCondizioneN=2} we get
\[ \int_{\{|x|\in[\frac{1}{K},K]\}}\widehat{\phi}(x)^2 dx\geq \frac{m}{K^2} >0,\]
thus proving \eqref{finonnulla}.
\\
We pass to the limit as $p\rightarrow +\infty$ into \eqref{equaziocappuccio} as follows:
by Lemma \ref{lemma:beta1pbounded} and Lemma \ref{lemma:primastimabetapN=2} there exists $\widetilde\beta_1< 0$ such that up to a subsequence
\[
\widetilde\beta_1(p)\to \widetilde\beta_1\qquad\mbox{as $p\to+\infty$.}
\]

By \eqref{weakD}
\[
\int_{\R^2 \setminus\{0\}}\nabla \widehat{\phi_{{p}}^-}(x)\,\nabla\rho(x)\ dx
\ \rightarrow
\int_{\R^2 \setminus\{0\}}\nabla \widehat{\phi}(x)\,\nabla\rho(x)\ dx \ \ \mbox{ as }\ p\rightarrow +\infty.
\]
By \eqref{weakPeso}
\begin{equation}\label{convdeboleL2peso}
\int_{\R^2 \setminus\{0\}}\frac{\widehat{\phi_{{p}}^-}(x)\,\rho(x)}{|x|^2}dx\rightarrow \int_{\R^2 \setminus\{0\}}\frac{\widehat{\phi}(x)\,\rho(x)}{|x|^2}dx\ \ \mbox{ as }\ p\rightarrow +\infty.
\end{equation}
Last we show that
\[
\int_{\R^2 \setminus\{0\}}V_{p}^-(x)\,\widehat{\phi_{{p}}^-}(x)\,\rho(x)\ dx\ \rightarrow  \int_{\R^2 \setminus\{0\}}V^-(x)\,\widehat{\phi}(x)\,\rho(x)\ dx\ \ \mbox{ as }\ p\rightarrow +\infty,
\]
 indeed:
\begin{eqnarray*}
&&\left|\int_{\R^2 \setminus\{0\}}V_{p}^-(x)\,\widehat{\phi_{{p}}^-}(x)\,\rho(x)\ dx\ - \int_{\R^2\setminus\{0\}}V^-(x)\,\widehat{\phi}(x)\,\rho(x)\ dx\right|\leq
\\
&&
\qquad \leq \sup_{\rm{supp } (\rho)} \left( |x|^2|V_{p}^-(x)-V^-(x)| \right) \int_{\R^2 \setminus\{0\}} \frac{\widehat{\phi_{{p}}^-}(x)\,|\rho (x)|}{|x|^2}\ dx\
+ \ \left|    \int_{\R^2 \setminus\{0\}} \frac{    [\widehat{\phi_{{p}}^-}(x) - \widehat\phi(x)]\overbrace{|x|^2V^-(x) \rho(x)}^{:=\widetilde\rho(x)}}{|x|^2} dx\right|\\
&&
\qquad \leq \sup_{\rm{supp } (\rho)}\left( |x|^2|V^-_{p}(x)-V^-(x)|\right) \ C_{\rho}\ \| \frac{\widehat{\phi_{{p}}^-}}{|x|}\|_{L^2(\R^2)} \
+ \  \left|    \int_{\R^2 \setminus\{0\}} \frac{    [\widehat{\phi_{{p}}^-}(x) - \widehat\phi(x)]\widetilde\rho (x)}{|x|^2} dx\right|\\
&&\qquad\longrightarrow 0 \ \ \ \ \ \mbox{ as $p\rightarrow +\infty$,}
\end{eqnarray*}
where for the first term we have used the convergence result in \eqref{V^-_ptoV^-} and the bound in \eqref{normaPesata1}, while for the second term the convergence follows from \eqref{convdeboleL2peso} since $\widetilde\rho:=\rho |x|^2V^-(x) \in C^{\infty}_0(\R^2\setminus\{0\})$.\\
As a consequence by passing to the limit into \eqref{equaziocappuccio} we get
\begin{equation}
\int_{\R^2 \setminus\{0\}} \nabla \widehat{\phi}(x)\, \nabla\rho(x)\ dx
-
\int_{\R^2 \setminus\{0\}} V^-(x)\,\widehat{\phi}(x)\, \rho(x)\ dx
-
\widetilde{\beta_1} \int_{\R^2 \setminus\{0\}}
\frac{\widehat\phi(x)\,\rho(x)}{|x|^2}\ dx  =0,
\end{equation}
for any $\rho\in C^{\infty}_0 (\R^2\setminus\{0\})$,
namely $\widehat\phi$ is a (weak and so classical) nontrivial nonnegative solution to the limit equation
\[
-\widehat{\phi}''(s)-\frac{\widehat{\phi}'(s)}{s}-V^-(s)\widehat{\phi}(s)=\widetilde\beta_1\frac{\widehat{\phi}(s)}{s^2}\qquad s\in(0,+\infty),
\]
where $V^-(s)=\frac{2(\gamma+2)^2\delta^{\gamma+2}s^{\gamma}}{(\delta^{\gamma+2}+s^{\gamma+2})^2}$ is the function given by the convergence result in \eqref{V^-_ptoV^-}.\\
Reasoning as in \cite{GladialiGrossiNeves} and setting, for $s\in(0,+\infty)$, $\eta(s):=\widehat{\phi}(\delta (\frac{s}{2\sqrt2})^{\frac2{2+\gamma}})$ we then have that $\eta$ satisfies
\[
-\eta''(s)-\frac{\eta'(s)}{s}-\frac{1}{(1+\frac18 s^{2})^2}\eta(s)=\frac{4\widetilde\beta_1}{(\gamma+2)^2}\frac{\eta(s)}{s^2}\qquad s\in(0,+\infty).
\]
with $\frac{4\widetilde\beta_1}{(\gamma+2)^2}<0$ and thus by Proposition \ref{unicheSoluzioniDiEquazioneLimite}
\[\frac{4\widetilde\beta_1}{(\gamma+2)^2}=-1.\]
Hence the definition of $\gamma$ in \eqref{varie} implies
\[
\widetilde\beta_1=-\frac{\ell^2+2}2.
\]
The assertion follows considering the approximated value of $\ell\approx 7.1979$ (see \eqref{varie}).
\end{proof}

\

\

\

\section{Proof of Theorem \ref{teoPrincipale}}\label{section:ProofMain}

\

This section is devoted to the proof of Theorem \ref{teoPrincipale}.

\

\begin{proof}

As already done in Section \ref{section:analysis} (see \eqref{np}) we set for $p\in (1,+\infty)$
\[
n_p:=\max\{ n'_p, n''_p,[(\ep_p^+)^{-2}]+1\}\qquad\mbox{and}\qquad\widetilde{\beta_i}(p):=\widetilde\beta_i^{n_p}(p)\quad\mbox{for any $i\in\mathbb N^+$.}
\]
By Proposition \ref{proposition:MorseRadialeConPeso} a) to determine $m(u_p)$ is equivalent to count the number  $\widetilde{k_p^{n_p}}$ of the negative eigenvalues $\widetilde{ \mu_i}^{n_p}(p)$ of the operator $\widetilde{L_p^{n_p}}$ defined in \eqref{weightedOp}.
\\
Hence it is enough to show that
\begin{equation}\label{tesissima}
\widetilde{k_p^{n_p}}=12 \qquad\mbox{for $p$ sufficiently large.}
\end{equation}
From now on we simplify the notation as follows
\[\widetilde{\mu_i}(p):=\widetilde\mu_i^{n_p}(p)\quad\mbox{for any $i\in\mathbb N^+$.}
\]
By Lemma \ref{lemma:decompositionOfTheSpectrum} we have that
\begin{equation}
\sigma(\widetilde{L_p^{n_p}})=\sigma(\widetilde{L_{p,}^{n_p}}_{rad})+\sigma(-\Delta_{S^{1}})
\end{equation}
namely
the eigenvalues $\widetilde{\mu_j}(p)$ of $\widetilde{L_p^{n_p}}$ are given by
\begin{equation}\label{decomposizioneAutovaloriCasoMio} \widetilde{\mu_j}(p)\ =\ \widetilde{\beta_i}(p)\ +\ \lambda_k, \ \ \mbox{ for } i,j=1,2,\ldots,\ \ k=0,1, \ldots
\end{equation}
where  $\widetilde{\beta_i}(p)$, $i=1,2,\ldots$ are the eigenvalues of the {\emph{radial}} operator $\widetilde{L_p^{n}}_{rad}$ and $\lambda_k$, $k=0,1,\ldots$ are the eigenvalues of the Laplace-Beltrami operator $-\Delta_{S^{1}}$ on the unit sphere $S^{1}$. Recall that
\[\lambda_k=k^2 \ (\geq 0),\ \ k=0,1, \ldots \]
and that the eigenspace associated to $\lambda_0$ has dimension $1$ while the eigenspace associated to $\lambda_k$ has dimension $2$ (see \eqref{lambdak} and \eqref{multiplicity}).

By Corollary \ref{Cor:hpmorsetradotte}-b) we know that
$\widetilde{\beta_1}(p)\leq\widetilde{\beta_2}(p)<0\leq\widetilde{\beta_3}(p)<\ldots$, then
\[\widetilde{\beta_i}(p)\ +\ \lambda_k\ \geq 0 \ \ \mbox{ for  }i=3,4,\ldots, \ \mbox{ and }\ k=0,1,\ldots,\]
namely $\widetilde{\beta_i}(p)$, $i=3,4,\ldots $ do not give any contribution to the Morse index.

\

Next we study the remaining cases $\widetilde{\beta_i}(p)$, $i=1,2$.

\

About $\widetilde{\beta_2}(p)$, by Proposition \ref{LemmaStimeAutovaloriRadialeConPeso Iparte} we know that $\ \widetilde{\beta_2}(p)> -1\ $
and this implies that
\[\widetilde{\beta_2}(p)\ +\ \lambda_h\ >0 \ \ \mbox{ for }\ h=1,2,\ldots\]
while from Corollary \ref{Cor:hpmorsetradotte}-b) we have
\begin{equation}\label{primoautoneg}
\widetilde{\beta_2}(p)+\lambda_0=\widetilde{\beta_2}(p)<0.
\end{equation}
This gives one negative eigenvalue of $\widetilde{L_p^{n_p}}$
recalling that $\lambda_0=0$ has multiplicity $1$.

\

Let us now consider $\widetilde{\beta_1}(p).$
%Again by Corollary \ref{Cor:hpmorsetradotte}-b) we have \begin{equation}\label{secondoautoneg}
%\widetilde{\beta_1}(p)+\lambda_0=\widetilde{\beta_1}(p)<0
%\end{equation}
%and this gives, as before, one more negative eigenvalue of $\widetilde{L_p^{n_p}}$, which is the first radial eigencvalue.
%
%
%Observe that it must necessarily be for any $p\in (1,+\infty)$
%\begin{equation}\label{stimadx}
%\widetilde{\beta}_1(p)< -1.
%\end{equation}
%Otherwise we would have
%\[\widetilde{\beta}_1(p) +\lambda_h \geq 0,  \ \ \mbox{ for }\ h=1,2,\ldots\]
%Then, since the only negative eigenvalues of $\widetilde{L_p^{n_p}}$ would be the two ones in \eqref{primoautoneg} and \eqref{secondoautoneg},  we would then conclude that $\widetilde{k_p^{n_p}}=m(u_p)=2$
%thus contradicting  the result in Lemma \ref{LemaAftalionPacella} from which we already know that $m(u_p)\geq 4$.
%\\
%In order to count the remaining negative eigenvalues of  $\widetilde{L_p^{n_p}}$ given by \eqref{decomposizioneAutovalori} for $i=1$.\\

By Theorem \ref{thm:betaplimiteN=2} we know that
\[
\widetilde{\beta_1}(p)\rightarrow  -\frac{\ell^2+2}{2}\simeq -26.9 \qquad\mbox{ as }\ p\rightarrow +\infty,
\]
where $\ell$ is defined in \eqref{varie}. Therefore, for $p$ large
\[-\lambda_6=-36<\widetilde{\beta}_1(p)<-25=-\lambda_5\]
and as a consequence
\[
\widetilde{\beta_1}(p)\ +\ \lambda_k\ >0,\quad \ \ k=6,7,\ldots\\
\]
while
\begin{equation}\label{altriTANTIeigenvalues}
\widetilde{\beta_1}(p)\ +\ \lambda_k\ <0,\quad \ \ k=0,1,2,3,4,5.\\
\end{equation}
We know that the multiplicity of $\lambda_k$ is $1$ when $k=0$ and it is $2$ when $k\neq 0$, hence \eqref{altriTANTIeigenvalues} gives
$11$
negative eigenvalues of $\widetilde{L_p^{n_p}}$ (the first of them is equal to $\widetilde{\beta_1}(p)$ and it is the first radial eigenvalue).
By combining this with \eqref{primoautoneg}  we hence get
\[\widetilde{k_p^{n_p}}=12\quad\mbox{for $p$ large}\]
and this concludes the proof.
\end{proof}

\

\

\

\appendix
\section*{Appendix}
\setcounter{section}{1}

\

\begin{lemma}\label{lemma:appendixN3}
Let $N\geq3$ and $\eta\in C^2( \R^N\setminus\{0\})\cap D_{rad}(\R^N)$, then:
\[
|x|^{N-1}\eta(x)\to0\quad\textrm{as $|x|\to0$}\qquad\textrm{and}\qquad\frac{\eta(x)}{|x|}\to0\quad\textrm{as $|x|\to+\infty$}.
\]
\end{lemma}
\begin{proof}
Let $w$ be the Kelvin transform of $\eta$
\[
w(x):=|x|^{2-N}\eta(\frac{x}{|x|^2}),\qquad x\in\R^N\setminus\{0\}.
\]
We have that $w\in D^{1,2}_{rad}(\R^N)$, indeed
\begin{eqnarray*}
&&\int_{\R^N}|\nabla w(x)|^2 dx=\\
&=&N\omega_N\int_0^{+\infty}r^{N-1}\left[(2-N)^2r^{2-2N}\eta^2(\frac1r)+r^{-2N}\left(\eta'(\frac1r)\right)^2-2(2-N)r^{1-2N}\eta(\frac1r)\eta'(\frac1r)\right]dr\\
&=&N\omega_N\int_0^{+\infty}[(2-N)^2s^{N-1}\frac{\eta^2(s)}{s^2}+s^{N-1}(\eta'(s))^2-2(2-N)s^{\frac{N-1}{2}}\frac{\eta(s)}{s}s^{\frac{N-1}{2}}\eta'(s)]ds\\
&\leq&(2-N)^2\int_{\R^N}\frac{\eta(x)^2}{|x|^2}dx+\int_{\R^N}|\nabla\eta (x)|^2dx+2(N-2)\left(\int_{\R^N}\frac{\eta(x)^2}{|x|^2}dx\right)^{\frac12}\left(\int_{\R^N}|\nabla\eta (x)|^2dx\right)^{\frac12}<+\infty.
\end{eqnarray*}
Applying Strauss Lemma (see \cite{BerestyckiLions}) to $w$
\[
w(y)\leq\frac{C}{|y|^{\frac{N-2}{2}}}\qquad\textrm{for $y\neq 0$,}
\]
so
\begin{equation*}
|x|^{N-1}\eta(x)=|x|w(\frac{x}{|x|^2})\leq C |x|^{\frac N2}\to0\quad\textrm{as $|x|\to0$}.
\end{equation*}
On the other hand applying the Strauss Lemma directly to $\eta$ we get that in particular
\begin{equation*}
\frac{\eta(x)}{|x|}\to0\quad\textrm{as $|x|\to+\infty$}
\end{equation*}
and this concludes the proof.
\end{proof}

\

\

\begin{lemma}\label{lemma:appendixLinfinito}
Let $f\in L^{\infty}(\R^2)$, $f\geq0$ be such that $\frac1{|x|^4}f(\frac{x}{|x|^2})\in L^{\infty}(\R^2)$, let $\alpha\geq 0$ and let $\eta\in C^2( \R^2\setminus\{0\})\cap D_{rad}(\R^2)$,  $\eta\geq0$ be a radial nontrivial solution of
\begin{equation}
 -\lap \eta(x)-f(x)\eta(x)=-\alpha^2\frac{\eta}{|x|^2}\qquad x\in \R^2\setminus\{0\}
\end{equation}
Then
\[
|x|\eta(x)\to0\quad\textrm{as $|x|\to0$}\qquad\textrm{and}\qquad\frac{\eta(x)}{|x|}\to0\quad\textrm{as $|x|\to+\infty$}.
\]
\end{lemma}

\

\begin{proof}
The proof is inspired by \cite[Lemma 2.4]{GladialiGrossiNeves}.\\
In polar coordinates $\eta$ satisfies
\begin{equation}\label{equazioneLimiteAppendix}
 -\eta''-\frac{\eta'}{s}-f(s)\eta=-\alpha^2\frac{\eta}{s^2}\qquad s\in (0, +\infty)
\end{equation}
Let us observe that there exists $r_n\to0$ such that $r_n^{\alpha}\eta(r_n)=o(1)$ as $n\to+\infty$. This is trivial if $\alpha=0$, whereas if $\alpha>0$ such sequence does exist because, if not, we get $\eta(s)\geq\frac{C}{s^\alpha}$ in a neighborhood of $0$ and this contradicts $\int_0^{+\infty}\frac{\eta^2(s)}{s}ds<+\infty$, which holds true being $\eta\in D_{rad}(\R^2)$.\\

Let $R\in(0,1]$, using \eqref{equazioneLimiteAppendix} we have
\begin{eqnarray}\label{talpha+1}
\int_{r_n}^R t^{\alpha+1}f(t)\eta(t)\,dt&=&\int_{r_n}^R t^{\alpha+1}(-\eta''(t)-\frac{\eta'(t)}{t}+\alpha^2\frac{\eta(t)}{t^2})dt\noindent\\
&=&\int_{r_n}^R (-t^{\alpha+1}\eta'(t)+\alpha t^\alpha\eta(t))'dt\noindent\\
&=&-R^{\alpha+1}\eta'(R)+r_n^{\alpha+1}\eta'(r_n)+\alpha R^{\alpha}\eta(R)-\alpha r_n^{\alpha}\eta(r_n).
\end{eqnarray}
and since $f\in L^\infty(\R^2)$ and $\int_0^{+\infty}\frac{\eta^2(t)}{t}dt<+\infty$
\begin{equation}\label{t}
\int_{r_n}^1 t f(t)\eta(t)\,dt\leq C\int_{r_n}^1\frac{\eta(t)}{t^{\frac12}}dt
\leq C \left(\int_{r_n}^1\frac{\eta^2(t)}{t}dt\right)^{\frac12}\leq C.
\end{equation}

We now distinguish the case $\alpha>0$ from the case $\alpha=0$.

If $\alpha>0$, let us show that $r_n^{\alpha+1}\eta'(r_n)=o(1)$.
Multiplying equation \eqref{equazioneLimiteAppendix} by $t$ and integrating we get
\[
-\int_{r_n}^1\eta''(t)t\,dt=\int_{r_n}^1\eta'(t)\,dt-\alpha^2\int_{r_n}^1\frac{\eta(t)}{t}dt+\int_{r_n}^1 t f(t)\eta(t)\,dt,
\]
in the other hand integrating by parts
\[
-\int_{r_n}^1\eta''(t)t\,dt=-\eta'(1)+\eta'(r_n)r_n+\int_{r_n}^1\eta'(t)\,dt.
\]
Then
\begin{equation}\label{sonno}
-\eta'(1)+\eta'(r_n)r_n=-\alpha^2\int_{r_n}^1\frac{\eta(t)}{t}dt+\int_{r_n}^1 t f(t)\eta(t)\,dt,
\end{equation}
and multiplying by $r_n^{\alpha}$ we get
\begin{equation}\label{etaprimorn}
r_n^{\alpha+1}\eta'(r_n)=O(r_n^{\alpha})-\alpha^2 r_n^\alpha\int_{r_n}^1\frac{\eta(t)}{t}dt+r_n^\alpha\int_{r_n}^1 t f(t)\eta(t)\,dt.
\end{equation}
Since $\int_0^{+\infty}\frac{\eta^2(t)}{t}dt<+\infty$
\begin{eqnarray}\label{alpha2integrale}
r_n^\alpha\int_{r_n}^1\frac{\eta(t)}{t}dt&\leq& r_n^\alpha\left(\int_{r_n}^1\frac{\eta^2(t)}{t}dt\right)^{\frac12}\left(\int_{r_n}^1\frac 1t dt\right)^{\frac12}\nonumber\\
&\leq&r_n^\alpha C(-log(r_n))^{\frac12}
\end{eqnarray}
then by \eqref{etaprimorn}, \eqref{alpha2integrale} and \eqref{t} we get the claim: $r_n^{\alpha+1}\eta'(r_n)=o(1)$ and so in turn by \eqref{talpha+1}
\[
\int_{0}^R t^{\alpha+1}f(t)\eta(t)\,dt=-R^{\alpha+1}\eta'(R)+\alpha R^{\alpha}\eta(R).
\]
Then for any $s\in(0,1]$
\begin{eqnarray*}
\frac{\eta(s)}{s^{\alpha}}-\eta(1)&=&\int_s^1\left(-\frac{\eta'(R)}{R^{\alpha}}+\alpha\frac{\eta(R)}{R^{\alpha+1}}\right) dR\\
&=&\int_s^1\frac{1}{R^{2\alpha+1}}\left(\int_{0}^R t^{\alpha+1}f(t)\eta(t)\,dt\right) dR\\
&\leq&C\int_s^1\frac{1}{R^{2\alpha+1}}\left(\int_{0}^R t^{\alpha+\frac32}\frac{\eta(t)}{t^{\frac12}}\,dt\right) dR\\
&\leq&C\int_s^1\frac{1}{R^{2\alpha+1}}\left(\int_{0}^R t^{2\alpha+3}\,dt\right)^{\frac12}\left(\int_0^{R}\frac{\eta^2(t)}{t}dt\right)^{\frac12} dR\\
&\leq&C\int_s^1 R^{1-\alpha} dR
\end{eqnarray*}
At last
$$
\eta(s)\leq\left\{
             \begin{array}{ll}
               C s^{\alpha}, & \hbox{$\alpha<2$} \\
               C s^2 & \hbox{$\alpha>2$} \\
               C s^2|\log(s)| & \hbox{$\alpha=2$}
             \end{array}
           \right.
$$
so $s\eta(s)\to0$ as $s\to0$.

For what concerns the case $\alpha=0$, reasoning as above to derive \eqref{sonno} it is easy to see that
\[
\eta'(R)R=\int_{R}^1 t f(t)\eta(t)\,dt+\eta'(1),
\]
then for $s\in(0,1]$
\begin{eqnarray*}
\eta(s)-\eta(1)&=&-\int_s^1\eta'(R)dR=-\int_s^1\frac1R(R\eta'(R))dR\\
&=&-\int_s^1\eta'(R)dR=-\int_s^1\frac1R(\int_{R}^1 t f(t)\eta(t)\,dt+\eta'(1))dR\\
&\stackrel{f\geq0,\; \eta\geq0}{\leq}&C|\log(s)|,
\end{eqnarray*}
so also in this case $s\eta(s)\to0$ as $s\to0$.

\

Next let us consider $w(s)=\eta(\frac 1s)$. It is not hard to see that $w\in C^2( \R^2\setminus\{0\})\cap D_{rad}(\R^2)$ and it solves
 \begin{equation*}
 -w''-\frac{w'}{s}-\frac1{s^4}f(\frac1s)w=-\alpha^2\frac{w}{s^2}\qquad s\in (0, +\infty)\\
 \end{equation*}
So repeating the same reasoning as for $\eta$ and using that $\frac1{s^4}f(\frac1s)\in L^{\infty}((0,+\infty))$ we get that $sw(s)\to0$ as $s\to0$ and so $\frac{\eta(s)}s\to0$ as $s\to+\infty$ and this concludes the proof

\

It is worth to point out that actually if $\alpha>0$ the above estimates lead to a much stronger result, as for example $\eta\in L^\infty(\R^2)$.
\end{proof}

\

\begin{lemma}\label{lemma:appendixN=2}
Let $\Psi_{R,p}:A_p\to\R$ be the function defined in \eqref{psiN=2}, then
\[
\frac{\int_{A_p}|\nabla\Psi_{R,p}(y)|^2-p|\upp (y)|^{p-1}\Psi_{R,p}(y)^2 dy}{\int_{A_p}\frac{\Psi_{R,p}(y)^2}{|y|^2}dy}\leq-\frac{\ell^2+2}{2}(1+o_R(1)+o_p(1)).
\]
\end{lemma}

\

\begin{proof}
We set
\[\begin{array}{lr}
N_p:=\int_{A_p}|\nabla\Psi_{R,p}(y)|^2-p|\upp (y)|^{p-1}\Psi_{R,p}(y)^2 dy,\\ D_p:=\int_{A_p}\frac{\Psi_{R,p}(y)^2}{|y|^2}dy>0.
\end{array}
\]
Then, setting, for $0<a<b$, $A(a,b):=\{a<|y|<b\}$ we have:
\begin{equation}\label{Np su Dp}
\frac{N_p}{D_p}\leq\frac{\overbrace{\int_{A(\frac{\delta\varepsilon^-_p}{R},\delta R\varepsilon^-_p)}\!\!\!\!\!\!\!\!\!\!\!\! (|\nabla\Psi_{R,p}(y)|^2-p|\upp(y)|^{p-1}\Psi_{R,p}(y)^2) dy}^{=:N_{1,p}}+ \overbrace{\int_{A(\frac{\delta\varepsilon^-_p}{2R},\frac{\delta\varepsilon^-_p}{R})}\!\!\!\!|\nabla\Psi_{R,p}(y)|^2 dy}^{=:N_{2,p}}+ \overbrace{\int_{A(R\delta\varepsilon^-_p,2R\delta\varepsilon^-_p)}\!\!\!\!\!\!\!\!\!\!\!\!\!\!\!\!\!\!\!\!|\nabla\Psi_{R,p}(y)|^2 dy}^{=:N_{3,p}}}
{\underbrace{\int_{A(\frac{\delta\varepsilon^-_p}{R},\delta R\varepsilon^-_p)} \frac{\Psi_{R,p}(y)^2}{|y|^2}dy}_{=:D_{1,p}}+ \underbrace{\int_{A(\frac{\delta\varepsilon^-_p}{2R},\frac{\delta\varepsilon^-_p}{R})}\frac{\Psi_{R,p}(y)^2}{|y|^2}dy}_{=:D_{2,p}}+ \underbrace{\int_{A(R\delta\varepsilon^-_p,2R\delta\varepsilon^-_p)}\frac{\Psi_{R,p}(y)^2}{|y|^2}dy}_{=:D_{3,p}}}.
\end{equation}
Computing explicitly $N_{2,p}$ and $N_{3,p}$ we obtain:
\begin{equation}\label{N2p}
\frac{N_{2,p}}{2\pi}=\frac32\frac{(\frac1R)^{2+\gamma}}{(1+(\frac1R)^{2+\gamma})^{2}}\leq\frac{3}{2R^{2+\gamma}}
\end{equation}
and
\begin{equation}\label{N3p}
\frac{N_{3,p}}{2\pi}=3\frac{R^{2+\gamma}}{(1+R^{2+\gamma})^2}\leq\frac3{R^{2+\gamma}}.
\end{equation}
Furthermore we can also easily estimate $D_{1,p}$, $D_{2,p}$ and $D_{3,p}$ as follows:
\begin{equation}\label{D1p}
\frac{D_{1,p}}{2\pi}=\int_{\frac{\delta\varepsilon^-_p}{R}}^{\delta\varepsilon^-_p R}\frac{(\frac{r}{\delta\ep^-_p})^{2+\gamma}}{(1+(\frac{r}{\delta\ep^-_p})^{2+\gamma})^2}\frac1r dr\stackrel{t=(\frac{r}{\delta\varepsilon^-_p})^{2+\gamma}+1}{\leq}\frac1{2+\gamma}\int_{1+(\frac1R)^{2+\gamma}}^{1+R^{2+\gamma}}\frac{dt}{t^2} dt\leq\frac1{2+\gamma},
\end{equation}
\begin{equation}\label{D2p}
\frac{D_{2,p}}{2\pi}=\int_{\frac{\varepsilon^-_p\delta}{2R}}^{\frac{\varepsilon^-_p\delta}{R}}\frac{\psi_p^2(\frac{\delta\varepsilon^-_p}{R})}{(\frac{\delta\varepsilon^-_p}{R})^2}(r-\frac{\delta\varepsilon^-_p}{2R})^2\frac1r dr\leq\int_{\frac{\varepsilon^-_p\delta}{2R}}^{\frac{\varepsilon^-_p\delta}{R}}\frac{1}{R^{2+\gamma}}(\frac{2R}{\delta\varepsilon^-_p})^2(r-\frac{\delta\varepsilon^-_p}{2R})^2 \frac{2R}{\delta\varepsilon^-_p} dr=\frac1{3R^{2+\gamma}},
\end{equation}
\begin{equation}\label{D3p}
\frac{D_{3,p}}{2\pi}=\int_{R\varepsilon^-_p\delta}^{2R\varepsilon^-_p\delta}\frac{\psi_p^2(R\varepsilon^-_p\delta)}{(R\varepsilon^-_p\delta)^2}\frac{(r-2R\varepsilon^-_p\delta)^2}r dr\leq\int_{R\varepsilon^-_p\delta}^{2R\varepsilon^-_p\delta}\frac{1}{R^{2+\gamma}}\frac1{(R\varepsilon^-_p\delta)^3}(r-2R\varepsilon^-_p\delta)^2 dr=\frac1{3R^{2+\gamma}}.
\end{equation}

Let us now estimate $N_{1,p}$. In order to do so we define $\tilde\psi_p(s):=\psi_p(\delta\ep^-_p s)$, for $s\in[\frac 1R,R]$. Then (recalling that $s_p$ is defined as in \eqref{sp}) we have 
\begin{eqnarray}
\frac{N_{1,p}}{2\pi}&=&\int_{\frac1R}^{R}s\left((\tilde\psi'_p(s))^2-\left|\frac{\upp(\delta \ep^-_p s)}{\upp(s_p)}\right|^{p-1}(\tilde\psi_p(s))^2\right)ds\nonumber\\
&\underset{p\to+\infty}{\overset{\eqref{V^-_ptoV^-}}{\longrightarrow}}&\int_{\frac1R}^{R}s\left(\frac{(2+\gamma)^2}{4}\frac{s^{\gamma}(1-s^{2+\gamma})^2}{(1+s^{2+\gamma})^4}-\frac{2(2+\gamma)^2s^\gamma}{(1+s^{2+\gamma})^2}\frac{s^{2+\gamma}}{(1+s^{2+\gamma})^2}\right)ds\nonumber\\
&=&\frac{(2+\gamma)^2}{4}\left[\int_{\frac1R}^{R}\frac{((1+s^{2+\gamma})^2-12s^{2+\gamma})s^{1+\gamma}}{(1+s^{2+\gamma})^4}ds\right]\nonumber\\
&\overset{t=1+s^{2+\gamma}}{=}&\frac{2+\gamma}{4}\left[\int_{1+(\frac1R)^{2+\gamma}}^{1+R^{2+\gamma}}\frac{(t^2-12t+12)}{t^4}dt\right]\nonumber\\
&=&\frac{2+\gamma}{4}\left[-\frac{1}{1+R^{2+\gamma}}+\frac{1}{1+(\frac1R)^{2+\gamma}}+\frac{6}{(1+R^{2+\gamma})^2}-\frac{6}{(1+(\frac1R)^{2+\gamma})^2}+\right.\nonumber\\
& &\phantom{\frac{2+\gamma}{4}[[}\left.-\frac{4}{(1+R^{2+\gamma})^3}+\frac{4}{(1+(\frac1R)^{2+\gamma})^3}\right]\nonumber\\
&\leq&\frac{2+\gamma}{4}\left[1+O(\frac1{R^{2+\gamma}})\right].\nonumber
\end{eqnarray}
Then
\begin{equation}\label{N1p}
\frac{N_{1,p}}{2\pi}\leq-\frac{2+\gamma}{4}\left[1+O(\frac1{R^{2+\gamma}})+o_p(1)\right],
\end{equation}
which is negative for sufficiently large $R$ and $p$.

In conclusion, fixing $R$ sufficiently large, there exists $p_R$ such that for any $p\geq p_R$ we have (collecting \eqref{Np su Dp}, \eqref{N2p}, \eqref{N3p}, \eqref{D1p}, \eqref{D2p}, \eqref{D3p} and \eqref{N1p}):
\[
\frac{N_p}{D_p}\leq\frac{-\frac{2+\gamma}{4}(1+o_R(1)+o_p(1))}{\frac{1}{2+\gamma}(1+o_R(1))}=-\frac{(2+\gamma)^2}{4}(1+o_R(1)+o_p(1))\overset{\eqref{varie}}{=}-\frac{\ell^2+2}{2}(1+o_R(1)+o_p(1))
\]
\end{proof}

\

{\bf{Acknowledgements.}} F. De Marchis and I. Ianni acknowledge the support and the hospitality of the FIM at ETH Z\"urich in the spring 2015, where part of the research contained in this paper was carried out.

\

\end{document}